\documentclass[11pt,reqno,letterpaper]{amsart}
\usepackage{color}
\usepackage[colorlinks=true, allcolors=blue,backref=page]{hyperref}
\usepackage{amsmath, amssymb, amsthm}
\usepackage{mathrsfs}
\usepackage{mathtools}
\usepackage[noabbrev,capitalize,nameinlink]{cleveref}
\crefname{equation}{}{}
\usepackage{fullpage}
\usepackage[noadjust]{cite}
\usepackage{graphics}
\usepackage{pifont}
\usepackage{tikz}
\usepackage{bbm}
\usepackage[T1]{fontenc}

\usetikzlibrary{arrows.meta}

\usepackage{environ}
\usepackage{framed}
\usepackage{url}
\usepackage[linesnumbered,ruled,vlined]{algorithm2e}
\usepackage[noend]{algpseudocode}
\usepackage[labelfont=bf]{caption}
\usepackage{cite}
\usepackage{framed}
\usepackage[framemethod=tikz]{mdframed}
\tikzstyle{vertex}=[circle,draw=black,fill=black,inner sep=0,minimum size=0.35cm,text=white,font=\footnotesize]
\usepackage{appendix}
\usepackage{graphicx}
\usepackage[textsize=tiny]{todonotes}
\usepackage{tcolorbox}
\usepackage{enumerate}
\allowdisplaybreaks[1]
\usepackage{enumerate}
\usepackage{stmaryrd}
\usepackage[margin=1in]{geometry}

\counterwithin{figure}{section}

\usepackage[shortlabels]{enumitem}
\crefformat{enumi}{#2#1#3}
\crefrangeformat{enumi}{#3#1#4 to~#5#2#6}
\crefmultiformat{enumi}{#2#1#3}
{ and~#2#1#3}{, #2#1#3}{ and~#2#1#3}

\DeclareSymbolFont{symbolsC}{U}{pxsyc}{m}{n}
\SetSymbolFont{symbolsC}{bold}{U}{pxsyc}{bx}{n}
\DeclareFontSubstitution{U}{pxsyc}{m}{n}
\DeclareMathSymbol{\medcircle}{\mathbin}{symbolsC}{7}

\crefname{algocf}{Algorithm}{Algorithms}

\crefname{equation}{}{} %remove ``Equation''
\AtBeginEnvironment{appendices}{\crefalias{section}{appendix}} %appendices

\usepackage[color,final]{showkeys} %add in 'final' into parameter to remove showkeys

\colorlet{refkey}{orange!20}
\colorlet{labelkey}{blue!30}

\crefname{algocf}{Algorithm}{Algorithms}

\numberwithin{equation}{section}
\newtheorem{theorem}{Theorem}[section]

\newtheorem{lemma}[theorem]{Lemma}
\newtheorem{claim}[theorem]{Claim}

\crefname{claim}{Claim}{Claims}

\newtheorem{conjecture}[theorem]{Conjecture}
\newtheorem*{question*}{Question}

\theoremstyle{definition}
\newtheorem{definition}[theorem]{Definition}

\newtheorem*{definition*}{Definition}

\theoremstyle{remark}
\newtheorem*{remark}{Remark}
\newtheorem*{notation}{Notation}

\newcommand{\snorm}[1]{\lVert#1\rVert}

\newcommand{\bs}{\boldsymbol}
\newcommand{\mb}{\mathbb}
\newcommand{\mbf}{\mathbf}
\newcommand{\mbm}{\mathbbm}
\newcommand{\mc}{\mathcal}

\newcommand{\mr}{\mathrm}

\newcommand{\on}{\operatorname}

\newcommand{\eps}{\varepsilon}

\let\originalleft\left
\let\originalright\right
\renewcommand{\left}{\mathopen{}\mathclose\bgroup\originalleft}
\renewcommand{\right}{\aftergroup\egroup\originalright}

\allowdisplaybreaks

\newif\ifpublic
\ifpublic

\newcommand{\ignore}[1]{}

\else

\fi

\title{Substructures in Latin squares}

\author[Kwan]{Matthew Kwan}
\address{Institute of Science and Technology Austria (IST Austria), 3400 Klosterneuburg, Austria}
\email{matthew.kwan@ist.ac.at}

\author[A2]{Ashwin Sah}
\author[A3]{Mehtaab Sawhney}
\address{Department of Mathematics, Massachusetts Institute of Technology, Cambridge, MA 02139, USA}
\email{\{asah,msawhney\}@mit.edu}

\author[Simkin]{Michael Simkin}
\address{Center of Mathematical Sciences and Applications, Harvard University, Cambridge, MA 02138, USA}
\email{msimkin@cmsa.fas.harvard.edu}

\thanks{Sah and Sawhney were supported by NSF Graduate Research Fellowship Program DGE-1745302. Sah was supported by the PD Soros Fellowship. Simkin was supported by the Center of Mathematical Sciences and Applications at Harvard University.}

\begin{document}

\maketitle
\begin{abstract}
We prove several results about substructures in Latin squares. First,
we explain how to adapt our recent work on high-girth Steiner triple
systems to the setting of Latin squares, resolving a conjecture of Linial
that there exist Latin squares with arbitrarily high girth. As a consequence,
we see that the number of order-$n$ Latin squares with no \emph{intercalate}
(i.e., no $2\times2$ Latin subsquare) is at least $(e^{-9/4}n-o(n))^{n^{2}}$.
Equivalently, $\Pr\left[\mbf{N}=0\right]\ge e^{-n^{2}/4-o(n^{2})}=e^{-(1+o(1))\mb{E}\mbf{N}}$,
where $\mbf{N}$ is the number of intercalates in a uniformly random order-$n$
Latin square.

In fact, extending recent work of Kwan, Sah, and Sawhney, we resolve
the general large-deviation problem for intercalates in random Latin
squares, up to constant factors in the exponent: for any constant
$0<\delta\le1$ we have $\Pr[\mbf{N}\le(1-\delta)\mb{E}\mbf{N}]=\exp(-\Theta(n^{2}))$
and for any constant $\delta>0$ we have $\Pr[\mbf{N}\ge(1+\delta)\mb{E}\mbf{N}]=\exp(-\Theta(n^{4/3}\log n))$.

Finally, as an application of some new general tools for studying substructures in random Latin squares, we show that in almost all order-$n$ Latin squares, the
number of \emph{cuboctahedra} (i.e., the number of pairs of possibly
degenerate $2\times2$ submatrices with the same arrangement of symbols)
is of order $n^{4}$, which is the minimum possible. As observed by
Gowers and Long, this number can be interpreted as measuring ``how
associative'' the quasigroup associated with the Latin square is.
\end{abstract}

\section{Introduction}\label{sec:intro}
A \emph{Latin square} (of order $n$) is an $n\times n$ array filled
with the numbers $1$ through $n$ (we call these \emph{symbols}),
such that every symbol appears exactly once in each row and column.
Latin squares are a fundamental type of combinatorial design, and
in their various guises they play an important role in many contexts
(ranging, for example, from group theory, to experimental design,
to the theory of error-correcting codes). In particular, the multiplication
table of any group forms a Latin square. A classical introduction to the subject of Latin squares can be found in \cite{KD15}, though recently Latin squares have also played a role in the ``high-dimensional combinatorics'' program spearheaded by Linial, where they can be viewed as the first nontrivial case of a ``high-dimensional permutation''\footnote{To see the analogy to permutation matrices, note that a Latin square can equivalently, and more symmetrically, be viewed as an $n\times n\times n$ zero-one array such that every axis-aligned line sums to exactly 1.} (see for example \cite{LL14,LL16,LS18,Lin18}).

There are a number of surprisingly basic questions about Latin squares
that remain unanswered, especially with regard to statistical aspects.
For example, there is still a big gap between the best known upper
and lower bounds on the number of order-$n$ Latin squares (see for
example \cite[Chapter~17]{vW01}), and there is no known algorithm
that (provably) efficiently generates a uniformly random order-$n$ Latin square\footnote{Jacobson and Matthews~\cite{JM96} and Pittenger~\cite{Pit97} designed
Markov chains that converge to the uniform distribution, but it is
not known whether these Markov chains mix rapidly.}. Perhaps the main difficulty is that Latin squares are extremely
``rigid'' objects: in general there is very little freedom to make
local perturbations to change one Latin square into another.

Some of the most fundamental questions in this area concern existence
and enumeration of various types of substructures. We collect a few
different results of this type.

\subsection{Intercalates}\label{sub:intercalates}
Perhaps the simplest substructures one may wish to consider are \emph{intercalates}, which are order-2 Latin (combinatorial) subsquares. That is, an intercalate in a Latin square $L$ is a pair of rows $i<j$ and a pair of columns $x<y$ such that $L_{i,x}=L_{j,y}$ and $L_{i,y}=L_{j,x}$.
It is a classical fact that for all orders except 2 and 4 there exist
Latin squares with no intercalates~\cite{KLR75,KT76,McL75} (such
Latin squares are said to have property ``$N_{2}$''). As our first
result, we obtain the first nontrivial lower bound on the \emph{number}
of order-$n$ Latin squares with this property (upper bounds have
previously been proved in \cite{MW99,KSS21}).

\begin{theorem}\label{thm:pasch-free}
The number of order-$n$ Latin squares with
no intercalates is at least
\[
\left(e^{-9/4}n-o(n)\right)^{n^{2}}.
\]
\end{theorem}

\cref{thm:pasch-free} is proved by adapting our recent work on \emph{high-girth
Steiner triple systems}; we discuss this further in \cref{sub:intro-girth}.

For comparison, the total number of order-$n$ Latin squares is well
known\footnote{This is classical but nontrivial; it follows from celebrated permanent estimates due to Bregman~\cite{Bre73} and Egorychev--Falikman~\cite{Fal81,Ego81}.
See for example \cite[Chapter~17]{vW01}.} to be $(e^{-2}n-o(n))^{n^{2}}$. So, \cref{thm:pasch-free} can be
interpreted as the fact that a \emph{random }order-$n$ Latin square
is intercalate-free with probability at least $e^{-n^{2}/4-o(n^{2})}$.
Resolving a conjecture of McKay and Wanless~\cite{MW99} (see also \cite{CGW08,KS18}),
Kwan, Sah, and Sawhney~\cite{KSS21} recently proved\footnote{This is easy to guess heuristically but surprisingly difficult to prove; we are not aware of a way to estimate the expectation $\mb E \mbf N$ without going through large deviation estimates.} that the \emph{expected}
number of intercalates in a random order-$n$ Latin square is $n^{2}/4+o(n^2)$,
so writing $\mbf{N}$ for the number of intercalates in a random
order-$n$ Latin square, \cref{thm:pasch-free} corresponds to the
Poisson-type inequality
\[\Pr\left[\mbf{N}=0\right]\ge\exp\left(-(1+o(1))\mb{E}\mbf{N}\right).\]

Combining this with \cite[Theorem~1.2(a)]{KSS21}, we obtain an optimal lower-tail large deviation estimate, up to a constant factor in the exponent.
\begin{theorem}\label{thm:lower-tail}
Let $\mbf{N}$ be the number of intercalates
in a random order-$n$ Latin square, and fix a constant $0<\delta\le1$.
Then
\[\Pr\left[\mbf{N}\le(1-\delta)\mb{E}\mbf{N}\right]=\exp(-\Theta(n^{2})).\]
\end{theorem}

We expect that \cref{thm:pasch-free} is best-possible (related
conjectures appeared in \cite{GKLO20,BW19}). In \cref{conj:lower-tail}, we make a specific prediction for the constant factor in the exponent, in the setting of \cref{thm:lower-tail}.

The upper tails of $\mbf{N}$ behave rather differently, due to the ``infamous
upper tail'' phenomenon elucidated by Janson~\cite{JR02}. Specifically,
it seems that the ``most likely way'' for $\mbf{N}$ to be much
larger than its expected value is for $\mbf{N}$ to contain a ``tightly
clustered'' set of intercalates (for example, if $L$ contains the
multiplication table of an abelian 2-group $(\mb{Z}/2\mb{Z})^{k}$,
then it contains at least about $2^{3k}$ intercalates).

We are able to obtain a similarly optimal estimate for the upper tail,
using very different methods from \cref{thm:lower-tail} (specifically,
we adapt the machinery of Harel, Mousset, and Samotij~\cite{HMS19}, studying upper tails using high moments and entropic
stability).
\begin{theorem}\label{thm:upper-tail}
Let $\mbf{N}$ be the number of intercalates
in a random order-$n$ Latin square, and fix a constant $\delta>0$.
Then
\[\Pr\left[\mbf{N}\ge(1+\delta)\mb{E}\mbf{N}\right]=\exp(-\Theta(n^{4/3}\log n)).\]
\end{theorem}

\cref{thm:upper-tail} closes the gap between lower and upper bounds
recently proved by Kwan, Sah, and Sawhney in \cite{KSS21}. Actually, we
are able to prove an even sharper large deviation inequality for intercalates
in random \emph{Latin rectangles}, in terms of a certain extremal
function; see \cref{thm:large-deviations-rectangle}.

\subsection{Cycles and girth}\label{sub:intro-girth}
Recall that the girth of a graph is the length of its shortest cycle.
In \cite{Lin18}, Linial defines a \emph{cycle}\footnote{We remark that the word ``cycle'' in the context of Latin squares
sometimes refers to a different object: every pair of rows, columns,
or symbols defines a permutation which decomposes into ``row-cycles'',
``column-cycles'', or ``symbol-cycles''. We will not need this
notion in the present paper.} in a Latin square $L$ to be a set of rows $A$, a set of columns
$B$, and a set of symbols $C$, with $|A|+|B|+|C|>3$, such that the $A\times B$ subarray of $L$ contains at least $|A|+|B|+|C|-2$ symbols in the set $C$. He defines the \emph{girth} of a Latin square $L$ to be the minimum of $|A|+|B|+|C|$ over all such cycles in $L$. These definitions
are motivated by the \emph{Brown--Erd\H os--S\'os problem} in
extremal hypergraph theory, and in particular by an old conjecture
of Erd\H os on the existence of high-girth \emph{Steiner triple systems}, which we recently proved in \cite{KSSS22} (see \cite{KSSS22} for definitions and motivation). Answering a conjecture of Linial~\cite{Lin18},
we show that there exist Latin squares with arbitrarily high girth.

\begin{theorem}\label{thm:high-girth-latin-square}
Given $g\in\mb{N}$, there is $N_{\ref{thm:high-girth-latin-square}}(g)\in\mb{N}$ such that if $N \ge N_{\ref{thm:high-girth-latin-square}}(g)$, then there exists an order-$N$ Latin square with girth greater than $g$.
\end{theorem}

\cref{thm:high-girth-latin-square} can be proved in essentially the same way as the analogous theorem for Steiner triple systems in \cite{KSSS22}. The only real complication concerns a ``triangle-regularisation'' lemma, which is much simpler in the setting of \cite{KSSS22} than in the setting of \cref{thm:high-girth-latin-square}. Basically, we need a fractional triangle-decomposition result for quasirandom tripartite graphs. Suitable techniques in the dense tripartite setting have already been developed by Montgomery \cite{Mon17} and Bowditch and Dukes \cite{BD19}; we give a somewhat different proof suitable for our application (combining ideas from both these papers, and introducing some new ones).

\cref{thm:pasch-free} is closely related to \cref{thm:high-girth-latin-square}: it is not hard to see that a Latin square has girth greater than 6 if and only if it has no intercalate. The methods in \cite{KSSS22} (which we explain in this paper how to adapt to prove \cref{thm:high-girth-latin-square}) easily yield a lower bound on the \emph{number} of Latin squares or Steiner triple systems with girth greater than a given constant $g$ (see \cite[Theorem~1.3]{KSSS22}), so with a simple calculation one can deduce \cref{thm:pasch-free} from the proof of \cref{thm:high-girth-latin-square}.

\subsection{Cuboctahedra}\label{sub:cuboctahedra}
A \emph{cuboctahedron}\footnote{The reason for the name is that (as we will see in \cref{fig:cuboctahedron}), one can interpret these objects in such a way that they resemble geometric cuboctahedra. In \cite{GL20}, the authors simply call these objects ``octahedra''.} in a Latin square is a pair of pairs of rows $(r_1,r_2),(r_1',r_2')$ and a pair of pairs of columns $(c_1,c_2),(c_1',c_2')$ such that $L_{r_i,c_j}=L_{r_i',c_j'}$ for all $i,j\in \{1,2\}$. Essentially, this is a pair of $2\times 2$ submatrices with the same pattern of entries, though degeneracies are allowed (e.g., a pair of $1\times 2$ subrectangles with the same entries also counts, as does a pair of cells with the same entries). When discussing cuboctahedra, we will always refer to labeled cuboctahedra (unlike the case with intercalates).

Binary operations whose multiplication tables are Latin squares are called \emph{quasigroups} (roughly speaking, these are like groups without an associativity assumption).
It turns out that a quasigroup is a group if and only if its multiplication table has $n^5$ cuboctahedra (which is the maximum possible). This is sometimes called the \emph{quadrangle condition} (due to Brandt~\cite{Bra27}). As explored by Gowers and Long \cite{GL20}, the number of cuboctahedra in a Latin square is a measure of ``how associative'' its corresponding quasigroup is\footnote{We remark that a different measure of associativity was considered in \cite{Lis20,DW21}, though there seems to be no obvious connection between the two measures. We thank the anonymous referee for bringing this to our attention.}.

We show that a random Latin square typically has $(4+o(1))n^4$ cuboctahedra.

\begin{theorem}\label{thm:cuboctahedra}
A random order-$n$ Latin square $\mbf L$ has $(4+o(1))n^4$ cuboctahedra whp.
\end{theorem}

It is not hard to see that \emph{every} Latin square has at least $(3-o(1))n^4$ cuboctahedra\footnote{In fact there are always at least this many \emph{degenerate} cuboctahedra: there are always $(1-o(1))n^4$ cuboctahedra obtained by taking the same $2\times 2$ submatrix twice, there are $(1-o(1))n^4$ cuboctahedra obtained by taking pairs of $1\times 2$ subrectangles with the same pair of symbols, and there are $(1-o(1))n^4$ cuboctahedra obtained by taking pairs of  $2\times 1$ subrectangles with the same pair of symbols.}, so up to constant factors random Latin squares are typically ``as non-associative as possible''. It remains an interesting question whether there exist any Latin squares with $(4-\Omega(1))n^4$ cuboctahedra.

The upper bound in \cref{thm:cuboctahedra} (i.e., that almost every order-$n$ Latin square has at most $(4+o(1))n^4$ cuboctahedra) may be proved with similar methods to \cref{thm:upper-tail}. In fact, some aspects of the proof can be simplified substantially because we are not attempting to prove an optimal upper tail bound. We also state a general-purpose result that provides upper bounds on general configuration counts in random Latin squares, as long as a certain ``stability'' criterion is satisfied (\cref{thm:general-upper-tail}).

To prove the lower bound in \cref{thm:cuboctahedra}, we make use of some ideas developed in \cite{Kwa20,FK20,KSS21}, via which a random Latin square can be approximated by the so-called \emph{triangle removal process}. These ideas are subject to quantitative limitations of completion theorems due to Keevash~\cite{Kee18,Kee18c}, and are therefore only suitable for controlling events which occur with probability extremely close to 1 (specifically, they must hold with probability $1-\exp(-\Omega(n^{2-b}))$ for a very small constant $b$). We therefore require some non-standard arguments to prove very high probability bounds (see \cref{thm:strong-lower-cuboctahedra} for a precise statement).

\subsection{Further directions}\label{sub:intro-further}
There are a number of fascinating further directions of study
in this area. Concerning the constant factors in the exponents in \cref{thm:lower-tail,thm:upper-tail}, we make some fairly precise conjectures. Again let $\mbf{N}$ be the number of intercalates
in a random order-$n$ Latin square.
\begin{conjecture}\label{conj:lower-tail}
For every constant $0<\delta\le1$, we have
\[\Pr\left[\mbf{N}\le(1-\delta)\mb{E}\mbf{N}\right]=\exp\left(-(\delta+(1-\delta)\log(1-\delta)+o(1))\mb{E}\mbf{N}\right).\]
\end{conjecture}

\begin{conjecture}\label{conj:upper-tail}
For every constant $\delta>0$, we have 
\[\Pr\left[\mbf{N}\ge(1+\delta)n^{2}/4\right]=\exp\left(-(1+o(1))\Phi\left(\delta\mb{E}\mbf{N}\right)\log n\right),\]
where $\Phi(N)$ is the minimum number of nonempty cells in a partial
Latin square with at least $N$ intercalates.
\end{conjecture}

We do not know the asymptotic value of $\Phi(N)$ in general (though
we do know this asymptotic value is $(4N)^{2/3}$ when $(4N)^{1/3}$ is
a power of two; see \cref{thm:Phi}), and this would also be interesting to investigate
further. Relatedly, it would also be interesting to find the asymptotics
of the maximum possible number of intercalates in an order-$n$ Latin
square (see \cite{Bar13,BCW14} for the current best bounds on this problem).

Regarding Latin subsquares of order greater than 2: it was conjectured by Hilton (see \cite[Problem~1.7]{DK91}) that for sufficiently large $n$ there exist order-$n$ Latin squares with no proper Latin subsquare at all (such Latin squares are said to have property ``$N_{\infty}$''). This conjecture remains open when $n=2^{a}3^{b}$ for $a\ge1$ and $b\ge0$ (see \cite{MWW07} and the references therein). Although a random order-$n$ Latin square typically has about $n^{2}/4$ intercalates, Kwan, Sah, and Sawhney \cite{KSS21} conjectured that the number of $3\times3$ subsquares has an asymptotic Poisson distribution with mean $O(1)$, and that almost all Latin squares have no Latin subsquares of order 4 or higher (see also \cite[Section~10]{MW99}). That is to say, intercalates are the ``hardest'' type of Latin subsquare to avoid. In fact, it seems plausible that the random construction used to prove \cref{thm:high-girth-latin-square} may have property $N_{\infty}$ with probability $\Omega(1)$, but an attempt to prove this would require deconstructing the proof of \cite[Theorem~1.1]{KSSS22} to a far greater extent than we do in this paper.

Finally, we reiterate that it would be interesting to determine the asymptotic minimum number of cuboctahedra in an order-$n$ Latin square (i.e., the ``least associative'' it is possible for a Latin square to be).

\subsection{Outline}\label{sub:outline}
In \cref{sec:upper-reduction} we reduce \cref{thm:upper-tail} to a large deviation problem for random Latin rectangles. In \cref{sec:upper-tail-upper} we prove an upper bound on upper tail probabilities for intercalates in random Latin rectangles, and in \cref{sec:large-deviations-lower} we sketch how to prove a corresponding lower bound (this is not necessary for the proof of \cref{thm:upper-tail}, but may be of independent interest). In \cref{sec:extremal-intercalates} we prove some basic facts about a certain extremal function $\Phi$, which features in \cref{sec:upper-reduction,sec:upper-tail-upper,sec:large-deviations-lower}. In \cref{sec:few-rows} we make some observations about intercalates in random Latin rectangles with very few rows; here we observe quite different behavior.

In \cref{sec:cuboctahedra} we explain how the ideas used to prove \cref{thm:upper-tail} can be generalized to a wide range of different configurations other than intercalates. Based on this, we explain how to prove the upper bound in \cref{thm:cuboctahedra} (though there is some tedious casework that we omit). We also discuss the machinery for probabilistic transference between Latin squares and the triangle removal process, and use this machinery to prove the lower bound in \cref{thm:cuboctahedra}.

Finally, in \cref{sec:high-girth} we explain how to adapt our work in \cite{KSSS22} to prove \cref{thm:high-girth-latin-square,thm:pasch-free}. This section is mostly targeted towards readers who have read \cite{KSSS22} or at least its proof outline, though we do provide a high-level summary of the overall approach.

\subsection{Notation}
We use standard asymptotic notation throughout, as follows. For functions $f=f(n)$ and $g=g(n)$, we write $f=O(g)$ or $f \lesssim g$ to mean that there is a constant $C$ such that $|f|\le C|g|$, $f=\Omega(g)$ to mean that there is a constant $c>0$ such that $f(n)\ge c|g(n)|$ for sufficiently large $n$, and $f=o(g)$ to mean that $f/g\to0$ as $n\to\infty$. Subscripts on asymptotic notation indicate quantities that should be treated as constants. Also, following \cite{Kee14}, the notation $f=1\pm\varepsilon$ means
$1-\varepsilon\le f\le1+\varepsilon$.

For a real number $x$, the floor and ceiling functions are denoted $\lfloor x\rfloor=\max\{i\in \mb Z:i\le x\}$ and $\lceil x\rceil =\min\{i\in\mb Z:i\ge x\}$. We will however mostly omit floor and ceiling symbols and assume large numbers are integers, wherever divisibility considerations are not important. All logarithms in this paper are in base $e$, unless specified otherwise.

\subsection*{Acknowledgements}
We thank Freddie Manners for suggesting the enumeration of cuboctahedra in random Latin squares. We thank Zach Hunter for pointing out typographical mistakes as well as a minor error in the statement of \cref{lem:reg}.

\section{Upper tails for intercalates: reducing to Latin rectangles}\label{sec:upper-reduction}
\newcommand{\ic}{N}
In the setting of \cref{thm:upper-tail}, the lower bound $\Pr[\mbf N\ge (1+\delta)\mb E \mbf N]\ge \exp(-O(n^{4/3}\log n))$ already appears as \cite[Theorem~1.1(d)]{KSS21}. So, to prove \cref{thm:upper-tail}, it suffices to prove the following upper bound.
\begin{theorem}\label{thm:large-deviations}
Fix a constant $\delta>0$. Let $\ic(\mbf{L})$
be the number of intercalates in a uniformly random order-$n$ Latin
square $\mbf{L}$. Then 
\[\Pr[\ic(\mbf{L})\ge(1+\delta)n^{2}/4]\le\exp\left(-\Omega_{\delta}(n^{4/3}\log n)\right).\]
\end{theorem}

We will deduce \cref{thm:large-deviations} from the following related theorem, which gives a \emph{sharp} upper tail estimate for random \emph{Latin rectangles}. A $k\times n$ Latin rectangle is a $k\times n$ array containing the symbols $\{1,\dots,n\}$, such that every symbol appears \emph{at most} once in each row and column.

\begin{theorem}\label{thm:large-deviations-rectangle}
Fix a constant $\delta>0$.
Let $\ic(\mbf{L})$ be the number of intercalates in a uniformly
random $k\times n$ Latin rectangle $\mbf{L}$, where $k=o(n)$
and $k=\omega((\log n)^{3/2})$. Let $\Phi(N)$ be the minimum number of nonempty cells in a partial Latin square with at least $N$ intercalates. Then 
\[\Pr\left[\ic(\mbf{L})\ge(1+\delta)\frac{k^{2}}4\right]=\exp\left(-(1+o(1))\Phi\left(\delta\frac{k^2}4\right)\log n\right).\]
\end{theorem}

We remark that the assumption $k=\omega((\log n)^{3/2})$ is in fact necessary; qualitatively different behavior occurs for smaller $k$ (we show this in \cref{thm:log-necessary}). We believe that the assumption $k=o(n)$ is unnecessary (and that the conclusion holds even for $k=n$, as in \cref{conj:upper-tail}) but we are unable to prove this.

The estimate in \cref{thm:large-deviations-rectangle} (approximately $n^{-\Phi(\delta k^2\!/4)}$) should be interpreted as being essentially the probability that a specific partial Latin square with $\Phi(\delta k^2\!/4)$ nonempty cells and $\delta k^2\!/4$ intercalates is present in $\mbf L$. If we condition on this event, then the conditional expected number of intercalates becomes about $(1+\delta)k^2\!/4$.

Regarding the value of $\Phi(N)$, its order of magnitude is always $N^{2/3}$, and we are able to determine $\liminf_{N\to \infty}\Phi(N)/N^{2/3}$. However, it seems plausible that the asymptotic value of $\Phi(N)$ may in general depend on number-theoretic properties of $N$. It would be interesting to investigate this further.
\begin{theorem}\label{thm:Phi}
Let $\Phi(N)$ be the minimum number of nonempty cells in a partial Latin square with at least $N$ intercalates.
\begin{enumerate}
    \item $\Phi(N)\ge(1+o(1))(4N)^{2/3}$ for all $N$.
    \item $\Phi(N)=(1+o(1))(4N)^{2/3}$ when ${(4N)}^{1/3}$ is a power of two.
    \item $\Phi(N)\le(4+o(1))(4N)^{2/3}$ for all $N$.
    \item $\Phi(N+\varepsilon N)-\Phi(N)=o(N^{2/3})$ if $\varepsilon = o(1)$.
\end{enumerate}
\end{theorem}
We will prove \cref{thm:Phi} in \cref{sec:extremal-intercalates}.

\cref{thm:large-deviations-rectangle} can be separated into an upper bound and a lower bound on ${\Pr[\ic(\mbf{L}) \ge (1+\delta)k^2\!/4]}$. Only the upper bound is necessary to prove \cref{thm:large-deviations}; we conclude this section with the deduction.

\begin{proof}[Proof of \cref{thm:large-deviations}, given the upper bound in \cref{thm:large-deviations-rectangle}]
Let $k=\varepsilon n$, and let $\mbf{L}_{\mr r}$ be a uniformly random $k\times n$
Latin rectangle, for $\varepsilon>0$ sufficiently small such that
\[
\Pr[\ic(\mbf{L_{\mr r}})\ge(1+\delta/2)k^{2}/4]\le\exp(-\Omega_{\delta}(k^{4/3}\log n))
\]
(such an $\varepsilon$ exists by the upper bound in \cref{thm:large-deviations-rectangle}, and \cref{thm:Phi}(1)). Now, if $\ic(\mbf{L})\ge(1+\delta)n^{2}/4$
then, by averaging, $\mbf{L}$ has a set of $k$ rows containing
at least 
\[
(1+\delta)\frac{n^{2}}4\left(\!\binom{k}{2}/\binom{n}{2}\!\right)
\]
intercalates. Let $\mbf{L}_{\le k}$ be the Latin rectangle consisting
of the first $k$ rows of $\mbf{L}$; by symmetry and a union bound we deduce 
\[
\Pr[\ic(\mbf{L})\ge(1+\delta)n^{2}/4]\le\binom{n}{k}\Pr[\ic(\mbf{L}_{\le k})\ge(1+\delta/2)k^{2}/4].
\]

Then, the desired result follows from \cite[Proposition~4]{MW99}
(which is a simple application of Bregman's inequality and the Egorychev--Falikman
inequality for permanents). Specifically, comparing between $\mbf{L}_{\le k}$ and $\mbf{L}_{\mr r}$, there is a multiplicative change of measure
of at most $\exp\big(O(n(\log n)^{2})\big)$ for any event.
\end{proof}

\section{Upper-bounding the upper tail in a random Latin rectangle}\label{sec:upper-tail-upper}
In this section we prove the upper bound in \cref{thm:large-deviations}, building on techniques developed by Harel, Mousset and Samotij~\cite{HMS19}. We first need some basic definitions and estimates.

\begin{definition}\label{def:partial-latin}
A \emph{partial Latin array} is an array of dimensions $k\times n$ with some cells filled with one of $n$ symbols, in which no symbol appears more than once in any row or column. For partial Latin arrays $S,S'$ with the same dimensions, we write $S\subseteq S'$ if $S$ and $S'$ agree on all cells where $S$ is nonempty. Let $|S|$ be the number of nonempty cells in $S$, and let $\ic(S)$ be the number of intercalates in $S$. We say that a partial Latin array is \emph{completable} if there is a complete Latin rectangle of the same dimensions containing it.
\end{definition}

\begin{lemma}\label{lem:comparison}
Suppose $k=o(n)$. Let $\mbf L$ be a uniformly random $k\times n$ Latin rectangle, and let $Q$ be a $k\times n$ completable partial Latin array with $\left|Q\right|=o(k\sqrt{n})$. Then 
\[\mb{E}[\ic(\mbf{L})\,|\,Q\subseteq\mbf{L}]\le\frac{k^{2}}{4}+\ic(Q)+o(k^{2}).\]
\end{lemma}

To prove \cref{lem:comparison} we need a few auxiliary lemmas.
\begin{lemma}\label{lem:comparison-simple}
Let $Q$ be an completable $k\times n$ partial Latin array and let $\mbf{L}$ be a uniformly random $k\times n$ Latin rectangle. Consider a row-column pair $(r,c)$, and write $\mbf{L}_{r,c}$ for the symbol in the corresponding cell of $\mbf{L}$. Write $|Q_r|$ for the number of nonempty cells in row $r$ of $Q$ and suppose that $2k+|Q_r|\le n/2$. For any $s\in\{1,\dots,n\}$, and any cell $(r,c)$ which is empty in $Q$, we have
\[\Pr[\mbf{L}_{r,c}=s\,|\,Q\subseteq\mbf{L}]\le\frac{1}{n}+O\left(\frac{k+|Q_{r}|}{n^{2}}\right).\]
\end{lemma}

\begin{proof}
Let $\mc{L}$ be the set of all $k\times n$ Latin rectangles $\mbf L$ for which $Q\subseteq\mbf{L}$, and let $\mc{L}_{s}\subseteq\mc{L}$ be the set of all such Latin rectangles $L$ for which $L_{r,c}=s$. Consider the auxiliary bipartite graph $\mc{H}$ with parts $A = \mc{L}_{s}$ and $B = \mc{L}$, where we put an edge between $L\in A$ and $L'\in B$ if $L'$ can be obtained from $L$ by swapping the contents of $L_{r,c}$ with some cell in row $r$. (The trivial swap is allowed; note that elements in $\mc{L}_s$ appear on both sides of the bipartition.)

In this auxiliary bipartite graph:
\begin{enumerate}
\item[(1)] every $L\in B$ has degree at most 1, and 
\item[(2)] every $L'\in A$ has degree at least $n-2(k-1)-|Q_{r}|$.
\end{enumerate}
To see (1), note that there is exactly one cell in row $r$ with symbol $s$ (swapping cell $(r,c)$ with this cell may or may not produce a Latin rectangle). To see (2), we consider all $n-|Q_{r}|$ possible swaps. Not all of these produce Latin rectangles: at most $k-1$ of the swaps bring a symbol into $(r,c)$ which already appears in column $c$, and at most $k-1$ of the swaps bring the symbol $x$ into a column which already contains it.

It follows from (1) and (2) that 
\[\Pr[\mbf{L}_{r,c}=s\,|\,Q\subseteq\mbf{L}] = \frac{\left|\mc{L}_{s}\right|}{\left|\mc{L}\right|}\le\frac{1}{n-2(k-1)-\left|Q_{r}\right|},\]
which implies the desired result since $2k+|Q_r|\le n/2$.
\end{proof}

\begin{lemma}\label{lem:expectation-count}
Let $Q$ be a $k\times n$ partial Latin array with $\left|Q\right|=o(k\sqrt{n})$. Let $\mbf{A}(Q)$ be the random array obtained from $Q$ by independently putting a uniformly random symbol from $\{1,\dots,n\}$ in each empty cell. Then the number of intercalates $\ic(\mbf{A}(Q))$ in it satisfies
\[\mb{E}[\ic(\mbf{A}(Q))]\le\frac{k^{2}}{4}+\ic(Q)+o(k^2).\]
\end{lemma}

\begin{remark}
We assume that $2\times 2$ submatrices composed all of the same element in $\mbf{A}(Q)$ are counted as intercalates, though this will not matter to us.
\end{remark}

\begin{proof}
For an intercalate in $\mbf{A}(Q)$, say that one of its four entries was \emph{forced} if it appeared in $Q$.
\begin{itemize}
\item The contribution to $\mb{E}[\ic(\mbf{A}(Q))]$ from intercalates with four forced entries is $\ic(Q)$. 

\item The contribution to $\mb{E}[\ic(\mbf{A}(Q))]$ from intercalates with zero forced entries is at most $\binom{k}{2}\binom{n}{2}/n^{2}\le\frac{k^{2}}{4}$. Indeed, for any $2\times2$ subarray of empty cells in $Q$, the probability they form an intercalate in $\mbf{A}(Q)$ is $n^2/n^4$.

\item The contribution to $\mb{E}[\ic(\mbf{A}(Q))]$ from intercalates with one forced entry is at most $|Q|(k-1)(n-1)/n^{2}=o(k^{2})$. Indeed, for a given nonempty entry in $|Q|$, there are at most $(k-1)(n-1)$ ways to choose an additional row and column yielding a $2\times2$
subarray of $Q$ with 3 empty cells. For each such choice, the probability an intercalate is present in $\mbf{A}(Q)$ is $n/n^3$.

\item The contribution to $\mb{E}[\ic(\mbf{A}(Q))]$ from intercalates with precisely two forced entries that are in the same row or column is at most $2|Q|^{2}(n-1)/n^{2}\le2|Q|^{2}/n=o(k^{2})$. Indeed, for a given pair of nonempty entries in the same row or column of $Q$, there are at most $n-1$ ways to extend to a $2\times2$ subarray with 2 empty cells. For each such choice, an intercalate is present in $\mbf{A}(Q)$ with probability $1/n^{2}$.

\item The contribution to $\mb{E}[\ic(\mbf{A}(Q))]$ from intercalates with at least two forced entries in different rows and columns, and at least one unforced entry, is at most $|Q|^{2}/n=o(k^{2})$. Indeed, for a given pair of nonempty entries in different rows and columns of $Q$ the probability the corresponding $2\times2$ subarray forms an intercalate in $\mbf{A}(Q)$ is at most $1/n$. 
\end{itemize}
The desired result follows.
\end{proof}

Now we are ready to prove \cref{lem:comparison}.

\begin{proof}[Proof of \cref{lem:comparison}]
Choose $\varepsilon = o(1)$ such that $|Q|^2/(k^2n) = o(\eps^2)$ (which is possible since $|Q| = o(k\sqrt{n})$). There are at most $|Q|/(\varepsilon n)$ ``bad'' rows of $Q$ which have more than $\varepsilon n$ nonempty cells. In the conditional probability space given $Q\subseteq\mbf{L}$, let $\mbf{Q}'$ be obtained from $Q$ by adding all entries of $\mbf{L}$ in bad rows, and let $\mc{S}$ be the support of $\mbf{Q}'$. Then
\begin{align*}
\mb{E}[\ic(\mbf{L})\,|\,Q\subseteq\mbf{L}]&=\sum_{Q'\in\mc{S}}\Pr[\mbf Q'=Q']\mb{E}[\ic(\mbf{L})\,|\,Q'\subseteq\mbf{L}]\\
&=\sum_{Q'\in\mc{S}}\Pr[\mbf Q'=Q']\left(\ic(Q')+\mb{E}[\ic(\mbf L)-\ic(Q')|Q'\subseteq\mbf{L}]\right)\\
&\le\sum_{Q'\in\mc{S}}\Pr[\mbf Q'=Q']\left(\ic(Q')+(1+O(\varepsilon+k/n))\mb{E}[\ic(\mbf{A}(Q'))-\ic(Q')]\right)\\
&\le\sum_{Q'\in\mc{S}}\Pr[\mbf Q'=Q']\ic(Q')+(1+O(\varepsilon+k/n))\left(\frac{k^{2}}{4}+o(k^{2})\right).
\end{align*}
Here, in order to establish the first equality we have used linearity of expectation over the possible intercalates in $\mbf{A}(Q')$ which do not already appear in $Q'$. For the subsequent inequality we have used \cref{lem:comparison-simple} and for the final inequality we used \cref{lem:expectation-count}.

So, it suffices to prove that for any outcome $Q'\in\mc{S}$ we have $\ic(Q')=\ic(Q)+o(k^{2})$. To see this, first note that we always have $|Q'|\le|Q|/\varepsilon$. An intercalate appearing in $Q'$ but not in $Q$ is determined by a bad row and an entry of $Q'$ (namely, consider a row containing an element of $Q'\setminus Q$ in the intercalate and consider the opposite corner of the $2\times 2$ subarray), so there are at most $(|Q|/(\varepsilon n))\cdot(|Q|/\varepsilon) = o(k^2)$ such.
\end{proof}

We now begin to adapt some of the ideas of Harel, Mousset and Samotij~\cite{HMS19}.
\begin{definition}\label{def:seed}
Say a $k\times n$ partial Latin array $Q$ using the symbols $\{1,\dots,n\}$
is a $(\delta,\varepsilon,C)$-\emph{seed} if 
\begin{itemize}
\item $Q$ is completable,
\item $|Q|\le Ck^{4/3}\log n$, 
\item $\mb{E}[\ic(\mbf{L})\,|\,Q\subseteq\mbf{L}]\ge(1+\delta-\varepsilon)k^{2}/4$.
\end{itemize}
\end{definition}
Roughly speaking, these conditions say that $Q$ has few nonempty cells, but its appearance in $\mbf{L}$ is likely to dramatically increase the expected number of intercalates.

From now on we assume $k = o(n)$ and $k = \omega((\log n)^{3/2})$. Fix some $\varepsilon>0$ which is small in terms of $\delta$, and let $C$ be large in terms of $\delta,\varepsilon$ (small and large enough to satisfy certain inequalities later in the proof; then at the end of the proof we will take $\varepsilon\to0^+$ very slowly).

For a partial Latin array $L$, let $Z(L)$ be the indicator for the event that there is no $(\delta,\varepsilon,C)$-seed $Q\subseteq L$. The rest of the proof of the upper bound in \cref{thm:large-deviations-rectangle} boils down to the following two claims.
\begin{claim}
\label{clm:seed-1} $\Pr[\ic(\mbf{L})\ge(1+\delta)k^{2}/4\emph{ and }Z(\mbf{L})=1]\le\exp\left(-\Omega_{\delta,\varepsilon}(Ck^{4/3}\log n)\right).$ 
\end{claim}

\begin{claim}
\label{clm:seed-2} $\Pr[Z(\mbf{L})=0]\le\exp\left(-(1+o(1))\Phi((\delta-3\varepsilon)k^2/4)\log n\right)$.
\end{claim}

Once these two are proven, taking $C$ sufficiently large in terms of $\delta,\varepsilon$ implies
\[\Pr[\ic(\mbf{L})\ge(1+\delta)k^2/4]\le\exp(-(1+o(1))\Phi((\delta-3\varepsilon)k^2/4)\log n).\]
Then, sending $\varepsilon\to 0^+$ slowly and using \cref{thm:Phi}(1) and \cref{thm:Phi}(4), we obtain the desired upper bound in \cref{thm:large-deviations-rectangle}.

\begin{proof}[Proof of \cref{clm:seed-1}]
Say a \emph{potential intercalate} is a $k\times n$ partial Latin array in which exactly four entries are nonempty, forming an intercalate. Say that a set of potential intercalates $\{I_1,\ldots,I_\ell\}$ are \emph{compatible} if (a) every pair agrees on all cells where they are both nonempty, in which case it makes sense to consider their \emph{union} $I_1\cup\cdots\cup I_\ell$, and (b) the union is completable. 

Let $\ell=\lfloor Ck^{4/3}\log n/4\rfloor$, and note that if $L'\subseteq L$ and $Z(L)=1$ then $Z(L')=1$ as well. We have 
\begin{align*}
\mb{E}[\ic(\mbf{L})^{\ell}\cdot Z(\mbf{L})] & \le\!\!\!\sum_{\substack{I_{1},\ldots,I_{\ell}\\
\text{compatible,}\\
Z\left(I_{1}\cup\dots\cup I_{\ell-1}\right)=1
}
}\!\!\!\Pr\left[I_{1}\cup\dots\cup I_{\ell}\subseteq\mbf{L}\right]\\
 &=\!\!\!\sum_{\substack{I_{1},\ldots,I_{\ell-1}\\
\text{compatible,}\\
Z\left(I_{1}\cup\dots\cup I_{\ell-1}\right)=1
}
}\!\!\!\Pr\left[I_{1}\cup\dots\cup I_{\ell-1}\subseteq\mbf{L}\right]\cdot\mb E\left[\ic(\mbf{L})\middle|I_{1}\cup\dots\cup I_{\ell-1}\subseteq\mbf{L}\right]\\
 & \le\!\!\!\sum_{\substack{I_{1},\ldots,I_{\ell-1}\\
\text{compatible,}\\
Z\left(I_{1}\cup\dots\cup I_{\ell-2}\right)=1
}
}\!\!\!\Pr\left[I_{1}\cup\dots\cup I_{\ell-1}\subseteq\mbf{L}\right]\cdot\big((1+\delta-\varepsilon)k^{2}/4\big).
\end{align*}
In the last step we use the definition of a seed: since $Z(I_1\cup\cdots\cup I_{\ell-1})=1$ we see $I_1\cup\cdots\cup I_{\ell-1}$ is not a $(\delta,\varepsilon,C)$-seed (despite satisfying the first two properties in the definition of a seed). We can iterate the above inequality $\ell$
times to deduce that 
\begin{align*}
\mb E[\ic(\mbf{L})^{\ell}\cdot Z(\mbf{L})] & \le\big((1+\delta-\varepsilon)k^{2}/4\big)^{\ell}.
\end{align*}
Recalling that $\ell=Ck^{4/3}\log n/4$, Markov's inequality
gives 
\begin{align*}
\Pr[\ic(\mbf{L})\ge(1+\delta)k^{2}/4\;\;\text{and}\;\;Z(\mbf{L})=1] & \le\frac{\mb{E}[\ic(\mbf{L})^{\ell}\cdot Z(\mbf{L})]}{\big((1+\delta)k^{2}/4\big)^{\ell}},
\end{align*}
from which the claimed bound follows. 
\end{proof}
Next, we turn to \cref{clm:seed-2}. It would be too lossy to
upper-bound $\Pr[Z(\mbf{L})=0]$ by simply taking a union bound
over all possible seeds $Q$. The next idea is to instead consider
``minimal'' subsets of seeds.
\begin{definition}\label{def:core}
Say a $k\times n$ partial Latin array $Q$ using the symbols $\{1,\dots,n\}$ is a $(\delta,\varepsilon,C)$-\emph{core} if 
\begin{itemize}
\item $Q$ is completable,
\item $|Q|\le Ck^{4/3}\log n$,
\item $\ic(Q)\ge(\delta-3\varepsilon)k^{2}/4$ (therefore $|Q|\ge \Phi((\delta-3\varepsilon)k^2/4)$),
\item For any $Q_{-}\subseteq Q$ obtained from $Q$ by emptying a single nonempty cell, we have 
\[\ic(Q)-\ic(Q_{-})\ge\frac{\varepsilon k^{2}/4}{Ck^{4/3}\log n}.\]
\end{itemize}
\end{definition}

Note that for every $(\delta,\varepsilon,C)$-seed $Q$, we have $|Q| = o(k\sqrt{n})$ and hence \cref{lem:comparison} shows that $\ic(Q)\ge(\delta-2\varepsilon)k^2/4$ (note this puts a bound on how fast $\varepsilon$ decays at the end of the argument). By iteratively emptying cells that violate the fourth condition, we can always obtain a $(\delta,\varepsilon,C)$-core $Q'\subseteq Q$.
So, to prove \cref{clm:seed-2} it suffices to upper-bound the probability of containing a core. To this end, we collect a few observations about cores.

From now on, it will be sometimes convenient to think of a $k\times n$
partial Latin array with symbols in $\{1,\dots,n\}$ as a $k\times n\times n$
zero-one array (with ``shafts'' in the third dimension being identified
with symbols), such that every row, column, and symbol has at most
one ``$1$''. That is to say, symbols really have basically the same
role as rows and columns.

\begin{claim}\label{clm:core-2}
If $k=\omega((\log n)^{3/2})$ the number of $(\delta,\varepsilon,C)$-cores with $m$ nonempty cells is at most $n^{o(m)}$.
\end{claim}
To prove \cref{clm:core-2} we need a few auxiliary observations.

\begin{lemma}\label{lem:core-limited-symbols}
There is $u_m\lesssim_{\delta,\varepsilon,C}mk^{-2/3}\log n$ such that every $(\delta,\varepsilon,C)$-core with $m$ nonempty cells contains at most $u_m$ nonempty rows, at most $u_m$ nonempty columns, and at most $u_m$ distinct symbols.
\end{lemma}
\begin{proof}
For any core $Q$, the fourth condition in the definition of a core
says that every nonempty cell in $Q$ participates in at least $w:=(\varepsilon k^{2}/4)/(Ck^{4/3}\log n)$ intercalates.
 
This implies in particular that every nonempty row of $Q$ contains at least $w$ nonempty entries, so $Q$ has at most $u_{|Q|}:=|Q|/w$ nonempty rows. A similar argument applies to columns and symbols.
\end{proof}
Now, say an \emph{ordered core} $\vec{Q}$ is a core $Q$ equipped
with an ordering $(r_{1},c_{1})\prec\dots\prec(r_{|Q|},c_{|Q|})$
of its nonempty cells. We write $\vec{Q}_{i}\subseteq\vec{Q}$ for
the partial Latin array obtained from $\vec{Q}$ by emptying all nonempty
cells except the first $i$ of them. We say that $\vec{Q}$ is \emph{good}
if for every $|Q|/(\log n)^2\le i\le |Q|$, there are at least $(i/|Q|)^{3}k^{2/3}/(\log n)^{2}$ intercalates in $\vec{Q}_{i}$ containing the cell $(r_{i},c_{i})$.

\begin{claim}\label{clm:mostly-good}
If $k\ge(\log n)^{15}$, then a $1-o(1)$ fraction of orderings of a $(\delta,\varepsilon,C)$-core $Q$ are good.
\end{claim}
\begin{proof}
Given a $(\delta,\varepsilon,C)$-core $Q$, define the random ordered core
$\vec{\mbf{Q}}$ by taking a uniformly random ordering of the nonempty cells
of $Q$. We wish to show that $\vec{\mbf{Q}}$ is good whp.

We will take a union bound over all $|Q|/(\log n)^2\le i\le |Q|$. So, fix such an $i$. Note that, given a choice of $(r_{i},c_{i})$, the remaining nonempty cells in $\vec{\mbf{Q}}_{i}$ comprise a uniformly random subset of $i-1$ other
nonempty cells of $Q$. It will be convenient to work with a closely related ``binomial'' random array: let $\mbf{Q}_{i,r,c}^\mr{bin}$ be obtained by starting with $Q$, and emptying each cell other than $(r,c)$ with probability $1-(i-1)/(|Q|-1)$. Any property that holds with probability $1-n^{-\omega(1)}$ for $\mbf{Q}_{i,r,c}^\mr{bin}$ also holds with probability $1-n^{-\omega(1)}$ for $\vec{\mbf{Q}}_{i}$ conditioned on the event $(r_{i},c_{i})=(r,c)$ (this follows from the so-called ``Pittel inequality''; see \cite[p.~17]{JLR00}). It will suffice to study $Q_{i,r,c}^\mr{bin}$.

Now, recall that in $Q$ there are at least $\Omega_{\delta,\varepsilon,C}(k^{2/3}/\log n)$ intercalates containing $(r,c)$. Each of these intercalates involves a disjoint set of three nonempty cells other than $(r,c)$, and is therefore present in $\vec{\mbf{Q}}_{i,r,c}^\mr{bin}$ with probability $((i-1)/(|Q|-1))^{3}$. These intercalates are disjoint from each other aside from the shared cell $(r,c)$. By $k\ge(\log n)^{15}$ and a Chernoff bound, the number of such intercalates in $\vec{Q}_{i}$ is at least $(i/|Q|)^{3}k^{2/3}/(\log n)^{2}$ with probability $1-n^{-\omega(1)}$.
\end{proof}

\begin{proof}[Proof of \cref{clm:core-2}]
Let $u=u_m\lesssim_{\delta,\varepsilon,C}mk^{-2/3}\log n$ be as in \cref{lem:core-limited-symbols}. We have $u = o(m)$ since $k = \omega((\log n)^{3/2})$. There are $\binom{k}{u}\binom{n}{u}^{2}\le n^{3u}=n^{o(m)}$ ways to choose sets of $u$ rows, columns, and symbols; we fix such a choice and count the good ordered $(\delta,\varepsilon,C)$-cores involving only those rows, columns and symbols (by \cref{lem:core-limited-symbols} a bound of $m!n^{o(m)}$ will suffice).

First, there are $\binom{u^2}m\le (eu^2/m)^m \le (O(m^2k^{-4/3}(\log n)^2)/m)^m\le(\log n)^{5m} = n^{o(m)}$ ways to choose a set of nonempty cells, and there are $m!$ ways to choose an ordering on these cells $(r_1,c_1),\ldots,(r_m,c_m)$. Therefore it suffices to show there are $n^{o(m)}$ ways to choose the symbols in these ordered cells to produce an ordered core.

If $k\le(\log n)^{15}$ we use the trivial bound that there are at most $u\le(\log n)^{O(1)}$ choices at each step. Overall, there are at most $(\log n)^{O(m)} = n^{o(m)}$ ordered cores on this ordered list of cells.

Otherwise, we bound the number of good ordered cores on this list of cells $(r_i,c_i)$. Then \cref{clm:mostly-good} will imply a bound on the number of ordered cores as desired. For the first $m/(\log n)^2$ cells, we still use the trivial bound that there are at most $u\le n$ choices. For $m\ge i\ge m/(\log n)^2$, given any choices for the previous cells, we observe that the number of ways to choose a symbol for the $i$th cell is at most
\[\frac{u}{(i/m)^{3}k^{2/3}/(\log n)^{2}}\le (\log n)^{11}.\]
Indeed, say a ``potential intercalate at step $i$ for symbol $s$'' is a set of three cells other than $(r_i,c_i)$ which have been filled in the previous $i-1$ steps and which would form an intercalate with $(r_i,c_i)$ if it were filled with the symbol $s$. There are at most $u$ potential intercalates at step $i$ in total (corresponding to the at most $u$ supported columns in our list of cells, say), and we must choose a symbol $s$ such that there are at least $(i/m)^{3}k^{2/3}/(\log n)^2$ potential intercalates for $s$.

It therefore follows that the number of good ordered cores on this list of ordered cells is bounded by $n^{m/(\log n)^2}(\log n)^{O(m)} = n^{o(m)}$ and we are done.
\end{proof}

Finally we prove \cref{clm:seed-2}, which completes the proof of the upper bound in \cref{thm:large-deviations-rectangle} as discussed earlier.

\begin{proof}[Proof of \cref{clm:seed-2}]
By \cite[Theorem~4.7]{GM90}, for any $k\times n$ partial Latin array $Q$ with at most $u = u_{Ck^{4/3}\log n}$ (from \cref{lem:core-limited-symbols}) nonempty columns, we have $\Pr[Q\subseteq \mbf{L}]\le (O(1/n))^{|Q|}$. For $m_0=\Phi((\delta-3\varepsilon)k^2/4)$, by \cref{clm:core-2} and the definition of $\Phi$ and cores, along with the fact that seeds contain cores, we have
\[
\Pr[Z(\mbf{L})=0]\le\sum_{m = m_0}^{Ck^{4/3}\log n}n^{o(m)}(O(1/n))^{m}\le \exp\left(-(1+o(1))\Phi((\delta-3\varepsilon)k^2/4)\log n\right),
\]
as desired.
\end{proof}

\section{Lower-bounding the upper tail in a random Latin rectangle}\label{sec:large-deviations-lower}
In this section we sketch how to prove the lower bound in \cref{thm:large-deviations-rectangle}. This is not necessary for the proof of \cref{thm:upper-tail} but may be of independent interest. We refer to some of the ideas in \cref{sec:upper-tail-upper}, which should be read first.
\begin{proof}[Proof sketch of the lower bound in \cref{thm:large-deviations-rectangle}]
Fix $\varepsilon>0$ and some $C > 0$, and let $Q$ be a partial Latin square with $(\delta+2\varepsilon)k^2/4$ intercalates and $|Q|=\Phi((\delta+2\varepsilon)k^2/4)$. As in the discussion directly after \cref{def:core} in \cref{sec:upper-tail-upper}, we can find a $(\delta,\varepsilon,C)$-core $Q'\subseteq Q$ with at least $(\delta+\varepsilon)k^2/4$ intercalates by iteratively removing elements that violate the fourth condition. Then by \cref{lem:core-limited-symbols}, we see that $Q'$ has at most $O_{\delta,\varepsilon,C}(k^{2/3}\log n)$ nonempty columns.

Now, \cite[Theorem~4.7]{GM90} says that $\Pr[M\subseteq\mbf L]$ is very close to $n^{-|M|}$ for any ``reasonably small'' partial Latin square $M$. It implies that
\[\Pr[Q'\subseteq\mbf L]\ge \left(\frac{1+o(1)}n\right)^{\Phi((\delta+2\varepsilon)k^2/4)}.\]
It also implies, with an easy second-moment calculation, that
\[\Pr\left[\ic(\mbf L)\ge(1+\delta)\frac{k^2}4\,\middle|\,Q'\subseteq \mbf L\right] = 1-o(1).\]
The desired result follows, taking $\varepsilon\to 0^+$ and using \cref{thm:Phi}(1,4).
\end{proof}

\section{Maximizing the number of intercalates}\label{sec:extremal-intercalates}
In this section we prove \cref{thm:Phi}. To prove \cref{thm:Phi}(1) we need the following extremal theorem, which follows
directly from the ``colored'' version of the Kruskal--Katona theorem
due to Frankl, F\"uredi, and Kalai~\cite{FFK88}.

\begin{theorem}\label{thm:FFK}
Let $\mc{F}$ be a 3-partite graph with at least $n^3$ triangles. Then
$\mc{F}$ has at least $3n^2$ edges.
\end{theorem}

\begin{proof}[Proof of \cref{thm:Phi}(1)]
Fix three disjoint sets $R,C,S$ with $|R|=|C|=|S|$, such that the rows, columns, and symbols of $Q$ lie in $R,C,S$ respectively. Consider the 3-uniform tripartite graph $G$ with vertex set $R\cup C\cup S$, obtained by adding a triangle between a row $r\in R$, column $c\in C$ and symbol $s\in S$ whenever the $(r,c)$-entry of $Q$ contains $s$. Note that apart from the $|Q|$ triangles we directly added to form $G$ (which are all edge-disjoint), there are at least $4\ic(Q)$ other triangles in $G$ (the four triangles corresponding to an intercalate form four of the eight faces of an octahedron, and the four additional triangles forming the other four faces are unique to that intercalate). Note that $G$ has $3|Q|$ edges, so the desired result follows from \cref{thm:FFK} with $n = \lfloor(4\ic(Q))^{1/3}\rfloor$.
\end{proof}

\begin{proof}[Proof of \cref{thm:Phi}(2--4)]
For (2--3) we simply consider the Latin square corresponding to the multiplication table of an abelian 2-group $(\mb Z/2\mb Z)^k$, where $2^k$ is the smallest power of 2 bigger than $(4N+N^{3/4})^{1/3}$. It is easy to show (see for example \cite{BCW14}) that this Latin square has $(2^k)^2(2^k-1)/4\ge N$ intercalates.

For (4) we observe that given a partial Latin square with $\Phi(N)$ nonempty cells and $N$ intercalates, it is always possible to increase the number of intercalates by at least $\varepsilon N$ by adding a disjoint copy of the Latin square corresponding to the multiplication table $(\mb Z/2\mb Z)^k$, where $2^k=\Theta((\varepsilon N)^{1/3})$ is appropriately chosen.
\end{proof}

\section{Latin rectangles with very few rows}\label{sec:few-rows}
Now, we show that the assumption $k=\omega((\log n)^{3/2})$ in \cref{thm:large-deviations-rectangle} is in fact necessary. Note that $k^2\le k^{4/3}\log n$ precisely when $k\le (\log n)^{3/2}$.

\begin{theorem}\label{thm:log-necessary}
Fix a constant $\delta>0$.
Let $\ic(\mbf{L})$ be the number of intercalates in a uniformly
random $k\times n$ Latin rectangle $\mbf{L}$, where $k=n^{o(1)}$ and $k = \omega(1)$. Then 
\[\Pr\left[\ic(\mbf{L})\ge(1+\delta)\frac{k^{2}}4\right]\ge\exp\left((\delta-(1+\delta)\log(1+\delta)-o(1))\frac{k^2}{4}\right).\]
\end{theorem}

The bound in \cref{thm:log-necessary} is essentially the probability that a Poisson distribution with mean $k^2/4$ is at least $(1+\delta)k^2/4$. We suspect that a matching upper bound holds whenever $k=o((\log n)^{3/2})$ and $k = \omega(1)$. For very slowly growing $k$ this follows from \cite[Theorem~3]{MW99}, which shows that for fixed $k$ the number of intercalates in a uniformly random $k\times n$ Latin rectangle has distribution limiting to $\mr{Poi}(k(k-1)/4)$ as $n\to\infty$.

\begin{proof}[Proof sketch of \cref{thm:log-necessary}]
Fix $\varepsilon>0$. Say a collection of $m := \lceil(\delta+2\varepsilon)k^2/4\rceil$ potential intercalates is \emph{good} if no pair of these intercalates shares a column, or symbol. An easy calculation shows that the number of good collections of intercalates is $(1-o(1))\binom{n^2(n-1)^2k(k-1)/4}{m}$ since $k = n^{o(1)}$.

Fix a $k\times n$ partial Latin array $Q$ obtained by taking the union of $m$ intercalates in a good collection. Let $\ic^\ast(\mbf L,Q)$ be the number of intercalates in $\mbf L$ using no entry of $Q$. An argument similar to the proof of \cref{lem:comparison} ignoring intercalates which have no entries in $Q$ shows
\[\mb{E}[\ic(\mbf{L})-\ic(Q)-\ic^\ast(\mbf{L},Q)|Q\subseteq\mbf{L}] = o(k^2).\]

Markov's inequality then shows
\begin{align*}&\Pr\left[(1+\delta)\frac{k^2}{4}\le\ic(\mbf L)\le (1+\delta+4\varepsilon)\frac{k^2}{4}\,\middle|\,Q\subseteq\mbf L\right]\\
&\qquad\qquad\ge\Pr\left[(1-\varepsilon)\frac{k^2}{4}\le\ic^\ast(\mbf L,Q)\le(1+\varepsilon)\frac{k^2}{4}\,\middle|\,Q\subseteq\mbf L\right]-o(1).
\end{align*}
An easy second-moment calculation using \cite[Theorem~4.7]{GM90} (which says that $\Pr[M\subseteq \mbf L]$ is very close to $n^{-|M|}$ for any ``reasonably small'' partial Latin square $M$), and counting similar to the proof of \cref{lem:expectation-count}, we find
\[\Pr\left[(1-\varepsilon)\frac{k^2}{4}\le\ic^\ast(\mbf L,Q)\le (1+\varepsilon)\frac{k^2}{4}\middle|Q\subseteq \mbf L\right]=1-o(1).\]
Then \cite[Theorem~4.7]{GM90} yields
\begin{align*}
&\Pr\left[Q\subseteq \mbf L\text{ and }(1+\delta)\frac{k^2}{4}\le\ic(\mbf L)\le (1+\delta+4\varepsilon)\frac{k^2}{4}\right]\\
&\qquad\qquad\qquad\qquad\quad=\Pr\left[Q\subseteq \mbf L\right]\Pr\left[(1+\delta)\frac{k^2}{4}\le\ic(\mbf L)\le(1+\delta+4\varepsilon)\frac{k^2}{4}\,\middle|\,Q\subseteq \mbf L\right]\\
&\qquad\qquad\qquad\qquad\quad\ge(1-o(1))\left(\frac{1-o(1)}{n^4}\right)^m.
\end{align*}

Let $\mbf X$ be the number of size-$m$ good collections of intercalates in $\mbf L$, so
\begin{align*}
\Pr[\ic(\mbf L)\ge(1+\delta)k^2/4]\binom{(1+\delta+4\varepsilon)k^2/4}{m}&\ge \mb E[\mbf X\mbm 1_{(1+\delta)k^2/4\le\ic(\mbf L)\le (1+\delta+4\varepsilon)k^2/4}]\\
&\ge(1-o(1))\binom{n^2(n-1)^2k(k-1)/4}{m}\left(\frac{1-o(1)}{n^4}\right)^m
\end{align*}
by linearity of expectation.

The desired result follows, taking $\varepsilon\to 0^+$ slowly. Here we use the assumption $k=\omega(1)$, the approximation $\binom mq=\left((1+o(1))em/q\right)^q$ (which holds for $q=o(m)$ and $q=\omega(1)$) and the approximation $\binom mq = \exp((1+o(1))mH(q/m))$, for $H(t) = -t\log t-(1-t)\log(1-t)$ (which holds for $\min(q,m-q) = \Theta(m)$).
\end{proof}

\section{General configurations and cuboctahedra}\label{sec:cuboctahedra}
In this section we prove \cref{thm:cuboctahedra}. The upper bound and lower bound will be proved by quite different means, but for both we use the 3-uniform hypergraph formulation of a Latin square. A \emph{colored triple system} is a properly 3-colored 3-uniform hypergraph, where the color classes are labeled ``$R$'', ``$C$'' and ``$S$'' (short for ``rows'', ``columns'' and ``symbols'').
A colored triple system is \emph{Latin} if no pair of hyperedges intersect in more than one vertex. An order-$n$ \emph{partial Latin square} is a Latin colored triple system, where the color classes are
\[R=\{1,\dots,n\},\quad C=\{n+1,\dots,2n\}\quad S=\{2n+1,\dots,3n\}.\]
An order-$n$ \emph{Latin square} is a partial Latin square with exactly $n^2$ hyperedges.

\subsection{An upper bound on the number of cuboctahedra}\label{sub:cuboctahedra-upper}
For the upper bound we adapt the ideas used to prove \cref{thm:large-deviations}. In fact, we give a general high-probability upper bound for counts of configurations that satisfy a certain ``stability'' property. We make no attempt to obtain sharp tail estimates, so the proof of this upper bound is basically just a simpler version of the proof of \cref{thm:large-deviations}. We will therefore be very brief with the details.

\begin{definition}\label{def:stable}
Fix a Latin colored triple system $H$, and let $X_H(Q)$ be the number of (labeled) copies of $H$ in a colored triple system $Q$. Let $\mbf B_{n,p}$ be a random colored triple system with color classes $R,C,S$, where each possible hyperedge is present with probability $p$, and let $\mbf{B}_n = \mbf{B}_{n,1/n}$. Note that $\mb E[X_H(\mbf B_{n,p})]=(1+o(1))n^{v(H)}p^{e(H)}$, where $v(H)$ and $e(H)$ are, respectively, the numbers of vertices and hyperedges in $H$.

Say that $H$ is \emph{$\alpha$-stable} if $\mb{E}X_H(\mbf B_n)=O(n^\alpha)$ (i.e., $\alpha \ge v(H)-e(H)$), and there is $t=t(n)=\omega(n(\log n)^2)$ such that $\mb E[X_H(\mbf B_n)\,|\,Q\subseteq\mbf{B}_n]-\mb E[X_H(\mbf B_n)]=o(n^\alpha)$ for any Latin colored triple system $Q$ with at most $t$ triples.
\end{definition}
\begin{theorem}\label{thm:general-upper-tail}
Fix an $\alpha$-stable Latin colored triple system $H\!$, and let $\mbf L$ be a uniformly random order-$n$ Latin square. Then $X_H(\mbf L)\le \mb E X_H(\mbf B_n)+o(n^\alpha)$ whp.
\end{theorem}
\begin{proof}
Fix $\gamma>0$ to be chosen later. Let $R'\subseteq R,C'\subseteq C,S'\subseteq S$ be the first $\gamma n$ rows, columns, and symbols of $R,C,S$, respectively. Let $\mbf L_{\mr r}$ be a uniformly random Latin rectangle with rows indexed by $R'$ and columns and symbols indexed by $C,S$, and let $\mbf L_{\mr r}'$ be the Latin colored triple system obtained from $\mbf L_{\mr r}$ by deleting all columns and symbols except those in $C',S'$.

Let $t(\cdot)$ be the function certifying $\alpha$-stability (as in \cref{def:stable}). Consider a Latin colored triple system $Q$ with color classes $R',C',S'$ and at most $t(\gamma n)$ hyperedges.
By \cref{lem:comparison-simple}, for any such $Q$ we have
\[\mb E[X_H(\mbf L_{\mr r}')\,|\,Q\subseteq\mbf L_{\mr r}']\le (1+O(\gamma))\mb E[X_H(\mbf B_{\gamma n,1/n})\,|\,Q\subseteq \mbf B_{\gamma n,1/n}].\]
By $\alpha$-stability of $H$,
\begin{align*}
\mb E[X_H(\mbf B_{\gamma n,1/n})\,|\,Q\subseteq \mbf B_{\gamma n,1/n}]-\mb E X_H(\mbf B_{\gamma n,1/n})&=O_\gamma(\mb E[X_H(\mbf B_{\gamma n})\,|\,Q\subseteq \mbf B_{\gamma n}]-\mb EX_H(\mbf B_{\gamma n}))\\
&=o_\gamma(n^\alpha).
\end{align*}
Therefore
\begin{align*}
\mb E[X_H(\mbf L_{\mr r}')\,|\,Q\subseteq \mbf L_{\mr r}']&\le(1+O(\gamma))\mb{E}X_H(\mbf{B}_{\gamma n,1/n})+o_\gamma(n^\alpha)\\
&\le\mb{E}X_H(\mbf{B}_{\gamma n,1/n}) + O(\gamma)\cdot\gamma^{e(H)}\mb{E}X_H(\mbf{B}_{\gamma n}) + o_\gamma(n^\alpha)\\
&\le\mb{E}X_H(\mbf{B}_{\gamma n,1/n}) + O(\gamma)\cdot\gamma^{e(H)}\cdot O((\gamma n)^\alpha)+o_\gamma(n^\alpha)\\
&\le \mb{E}X_H(\mbf{B}_{\gamma n,1/n}) +O(\gamma^{e(H)+\alpha+1}n^\alpha),
\end{align*}
using the first property of $\alpha$-stability (that $\mb{E}X_H(\mbf B_n)=O(n^\alpha)$). Let $\ell=\lfloor t(\gamma n)/e(H)\rfloor$; a similar calculation to that in the proof of \cref{clm:seed-1} shows that
\[\mb{E}X_H(\mbf L_{\mr r}')^\ell\le (\mb E X_H(\mbf B_{\gamma n,1/n})+ O(\gamma^{e(H)+\alpha+1}n^\alpha))^\ell,\]
where we have noted that $n$ is sufficiently large with respect to $\gamma$. By Markov's inequality and the fact that $\mb{E}X_H(\mbf B_n)=O(n^\alpha)$, for a sufficiently large absolute constant $M$ we have
\begin{align*}
\Pr[X_H(\mbf L_{\mr r}')\ge \mb E X_H(\mbf B_{\gamma n,1/n})+M\gamma^{e(H) + \alpha + 1} n^\alpha]
& \le  \frac{\mb{E} X_H(\mbf L_{\mr r}')^\ell}{\left(\mb E X_H(\mbf B_{\gamma n,1/n})+M\gamma^{e(H) + \alpha + 1} n^\alpha\right)^\ell}\\
& \le \exp(-\omega_{\gamma}(n(\log n)^2)).
\end{align*}
Now, we finish similarly to the deduction of \cref{thm:large-deviations} from \cref{thm:large-deviations-rectangle}. First, using \cite[Proposition~4]{MW99} (i.e., Bregman's inequality and the Egorychev--Falikman inequality for permanents), any event that holds for $\mbf{L}_{\mr r}'$ with probability $1-\exp(-\omega(n(\log n)^2))$ will hold with similar probability for the restriction of $\mbf{L}$ to the rows, columns, and symbols $R',C',S'$. Thus by a union bound and symmetry, we see that whp $\mbf{L}$ has at most $\mb{E}X_H(\mbf B_{\gamma n,1/n})+M\gamma^{e(H) + \alpha + 1} n^\alpha$ copies of $H$ in any choice of $\gamma n$ rows, columns, and symbols. An averaging computation reveals that this property implies $X_H(\mbf L)\le (1+o(1)) \left( \mb E X_H(\mbf B_{n})+M\gamma^{e(H) + \alpha + 1-v(H)} n^\alpha \right)$. Note that $\alpha + 1 + e(H) - v(H)\ge 1$ and thus the desired result follows by taking $\gamma\to 0^+$ slowly.
\end{proof}

We now prove the upper bound in \cref{thm:cuboctahedra} (i.e., that almost every order-$n$ Latin square has at most $(4+o(1))n^4$ cuboctahedra) using \cref{thm:general-upper-tail}.

\begin{proof}[Proof of the upper bound in \cref{thm:cuboctahedra}]
We say that a cuboctahedron is \emph{nondegenerate} if its defining rows $r_1,r_2,r_1',r_2'$ are distinct, defining columns $c_1,c_2,c_1',c_2'$ are distinct, and the four entries of the form $L_{r_i,c_j}$ are distinct. The number of nondegenerate cuboctahedra in a Latin square $L$ is $X_H(L)$, where $H$ is a certain 8-hyperedge, 12-vertex colored triple system depicted on the left hand side of \cref{fig:cuboctahedron}. Note that $\mb E X_H(\mbf B_n)=(1-o(1))n^4$. We claim that $H$ is $4$-stable with $t(n)=n(\log n)^3$. Fix a Latin colored triple system $Q$ with at most $t(n)$ hyperedges, and for a copy of $H$ in $\mbf B_n$, say one of its 8 hyperedges is \emph{forced} if it appears in $Q$.

\begin{figure}
	\begin{centering}
	    $\vcenter{\hbox{
		\includegraphics{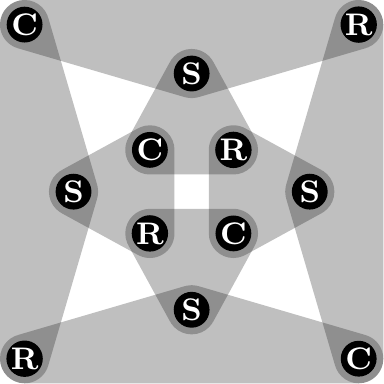}}}$
        $\qquad$
		$\vcenter{\hbox{
			\includegraphics{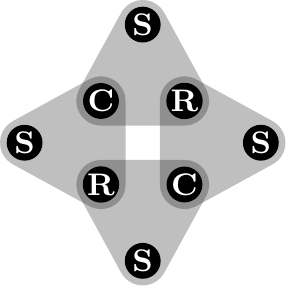}\quad
			\includegraphics{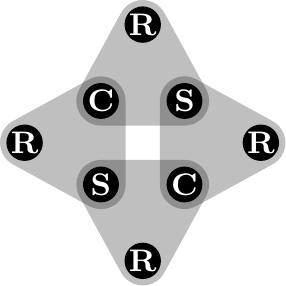}\quad
			\includegraphics{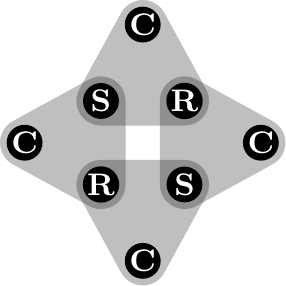}
			}}$
	\par\end{centering}
	\caption{\label{fig:cuboctahedron}On the left, a nondegenerate cuboctahedron (as an $\{R,C,S\}$-colored triple system). On the right, the three dominant degenerate cuboctahedra. In the $n\times n$ array formulation of a Latin square, the first of these degenerate cuboctahedra corresponds to a $2\times 2$ submatrix with distinct entries, taken twice. The second corresponds to a pair of distinct $2\times 1$ submatrices with the same pair of symbols, and the third corresponds to a pair of distinct $1\times 2$ submatrices with the same pair of symbols.}
\end{figure}

\begin{itemize}
	\item The contribution to $\mb E[X_H(\mbf B_n)\,|\,Q\subseteq \mbf B_n]$ from copies of $H$ with 1 forced hyperedge is $O(t n^9 n^{-7})=o(n^4)$, because there are $t$ ways to choose the forced entry, and $O(n^9)$ ways to choose the other 9 vertices to specify a copy of $H$. Then, the probability that all $7$ non-forced hyperedges are present is $n^{-7}$.
	\item The contribution from copies with 2 forced hyperedges is $O(t^2 n^7 n^{-6})=o(n^4)$ for similar reasons. Here we use the fact that every set of $2$ hyperedges in a cuboctahedron spans at least $5$ vertices.
	\item The contribution from copies with 3 forced hyperedges is $O(t^3 n^{12-7} n^{-5})=o(n^4)$, noting that every set of 3 hyperedges in a cuboctahedron spans at least $7$ vertices.
	\item For the contribution from copies with 4 forced hyperedges:
	\begin{itemize}
		\item If these 4 forced hyperedges are arranged in a ``4-cycle'', spanning 8 vertices, then the four forced hyperedges are determined by any three of them by the Latin property of $Q$, so the contribution is $O(t^3 n^{12-8} n^{-4})=o(n^4)$.
		\item Otherwise, the four forced hyperedges span at least 9 vertices, and the contribution is $O(t^4 n^{12-9} n^{-4})=o(n^4)$.
	\end{itemize}
	\item For the contribution from copies with 5 forced hyperedges we again distinguish cases: if the forced hyperedges contain a 4-cycle, then one can check that they are determined by some size-3 subset and the contribution is $O(t^3 n^{12-10} n^{-3})=o(n^4)$. Otherwise, at least 11 vertices are covered by $Q$ so the contribution is $O(t^5 n^{12-11} n^{-3})=o(n^4)$.
	\item For the contribution from copies with 6 forced hyperedges, there are three ``non-isomorphic'' cases to consider: the non-forced hyperedges can share a vertex, be at distance 1, or be on ``opposite sides'' of the cuboctahedron. In all cases, one can check that the forced hyperedges can be determined by some size-3 subset, and therefore compute that the contribution is $O(t^3 n^{12-11} n^{-2})=o(n^4)$.
	\item The contribution from copies with 7 forced hyperedges is $O(t^4 n^{-1})=o(n^4)$, as the forced hyperedges are determined by some size-3 subset.
	\item The contribution from copies with 8 forced hyperedges is $O(t^3)=o(n^4)$, as a cuboctahedron is determined by some 3 of its hyperedges.
\end{itemize}
	
The difference between $\mb E[X_H(\mbf B_n)\,|\,Q\subseteq \mbf B_n]$ and $\mb EX_H(\mbf B_n)$ arises from cuboctahedra with at least one forced hyperedge, so the above calculations show that $H$ is $4$-stable, as claimed. It follows that $X_H(\mbf L)\le (1+o(1))n^4$ whp, by \cref{thm:general-upper-tail}.

The total number of cuboctahedra (including degenerate ones) in $\mbf L$ can be expressed as a sum of quantities of the form $X_{H'}(\mbf L)$, where $H'$ ranges over a variety of Latin colored triple systems obtained by identifying vertices of $H$ in certain ways (respecting the Latin property).

For any such $H'$ with eight hyperedges (i.e., some vertices are identified but no hyperedges coincide), we have $\mb{E}X_{H'}(\mbf{B}_n) = O(n^{11})\cdot n^{-8} = O(n^3)$ and the exact same case structure as above shows that $H'$ is $4$-stable (whether it is $3$-stable is non-obvious but unnecessary for us). \cref{thm:general-upper-tail} then shows that $X_{H'}(\mbf{L})\le\mb{E}X_{H'}(\mbf{B}_n) + o(n^4) = o(n^4)$ whp.

Now we consider $H'$ with fewer than eight hyperedges. We will show that these $H'$ deterministically contribute $(3+o(1))n^4$ degenerate cuboctahedra. First, the dominant contribution comes from the three Latin colored triple systems depicted on the right of \cref{fig:cuboctahedron}. As discussed in \cref{sub:cuboctahedra}, the contribution from these diagrams is $(1+o(1))n^4$ each, for a total of $(3+o(1))n^4$ (with probability 1). Starting from these ``dominant'' cases, there are four further $H'$ that can be obtained by further identifying vertices (three with two hyperedges, and one with a single hyperedge). Each of these contribute only $O(n^3)$ to our count (again, with probability 1). As it turns out, one always obtains such a situation if any pair of faces is collapsed, other than ``opposite faces''.

It remains to consider degenerate cuboctahedra obtained by collapsing such ``opposite faces''. In the $n\times n$ array formulation of a Latin square, this corresponds to those cuboctahedra which consist of a pair of $2\times 2$ submatrices with the same arrangement of symbols, intersecting in exactly one entry (this is only possible if the arrangement has two of the same symbol). It is easy to see that in all such configurations there are three vertices which, if known, determine the entire configuration, so the contribution from such configurations is $O(n^3)$ with probability 1.

The desired result follows by adding up all the contributions from degenerate $H'$ (together with nondegenerate $H$).
\end{proof}

\subsection{The lower tail for cuboctahedra}\label{sub:cuboctahedra-lower}

We now turn to the lower bound in \cref{thm:cuboctahedra}. Recall the definition of a nondegenerate cuboctahedron from the proof of the upper bound in \cref{thm:cuboctahedra}. The contribution from degenerate cuboctahedra is always at least $(3-o(1))n^4$, so it suffices to prove the following strong lower tail bound for nondegenerate cuboctahedra.
\begin{theorem}\label{thm:strong-lower-cuboctahedra}
Fix $\delta > 0$, and let $X_H(\mbf{L})$ be the number of nondegenerate cuboctahedra in a uniformly random order-$n$ Latin square $\mbf L$. Then
\[\Pr[X_H(\mbf{L})\le (1-\delta)n^4]\le\exp(-\Omega_\delta(n^2)).\]
\end{theorem}

To prove \cref{thm:strong-lower-cuboctahedra}, we use some machinery from \cite{KSS21} (building on ideas in \cite{Kwa20,FK20}), which allows one to approximate a random Latin square with the so-called \emph{triangle-removal process}. We simply quote a number of statements which will be used in the proof; a more thorough discussion of the history of these techniques can be found in \cite{KSS21}. Let $\mc{L}_{m}$ be the set of order-$n$ partial Latin squares with $m$ hyperedges and let $\mc{L}$ be the set of order-$n$ Latin squares.

\begin{definition}[{cf.~\cite[Definition~2.2]{KSS21}}]\label{def:TRP}
The 3-partite \emph{triangle removal process} is defined as follows. Start with the complete 3-partite graph $K_{n,n,n}$ on the vertex set $R\cup C\cup S$.
At each step, consider the set of all triangles in the current graph,
select one uniformly at random, and remove it. After $m$
steps of this process, the set of removed triangles can be interpreted
as a partial Latin square $L\in\mc{L}_{m}$ unless we run out
of triangles before the $m$th step. Let $\mb{L}(n,m)$
be the distribution on $\mc{L}_{m}\cup\{\ast\} $ obtained
from $m$ steps of the triangle removal process, where ``$\ast$''
corresponds to the event that we run out of triangles.
\end{definition}
\begin{definition}[{\cite[Definition~2.3]{KSS21}}]\label{def:inheritance}
Let $\mc{T}_{m}\subseteq\mc{L}_{m}$ and $\mc{T}\subseteq\mc{L}$. We say that $\mc{T}_{m}$ is \emph{$\rho$-inherited} from $\mc{T}$ if for any $L\in\mc{T}$, taking $\mbf{L}_{m}\subseteq L$ as a uniformly random subset of $m$ hyperedges of $L$, we have $\mbf L_{m}\in\mc{T}_{m}$ with probability at least $\rho$.
\end{definition}

We will need the following transference theorem for inherited properties, which compares a subset of a uniformly random Latin square to the outcome of the triangle removal process. This theorem builds on a similar theorem for Steiner triple systems proved by Kwan~\cite{Kwa20}, using the work of Keevash~\cite{Kee18,Kee18c}, and the tripartite case is similar; see \cite{KSS21b}.
\begin{theorem}[{\cite[Theorem~2.4]{KSS21}}]\label{thm:transference-TRP}
Let $\alpha\in(0,1/2)$. There is an absolute constant $\gamma>0$ such that the following
holds. Consider $\mc{T}_{m}\subseteq\mc{L}_{m}$ with $m = \alpha n^2$ and $\mc{T}\subseteq\mc{L}$
such that $\mc{T}_{m}$ is $1/2$-inherited from $\mc{T}$.
Let $\mbf{P}\sim\mb{L}(n,m)$ be a partial
Latin square obtained by $m$ steps of the triangle removal process,
and let $\mbf{L}\in\mc{L}$ be a uniformly random order-$n$ Latin
square. Then
\[\Pr[\mbf{L}\in\mc{T}]\le\exp(O(n^{2-\gamma}))\Pr[\mbf{P}\in\mc{T}_m].\]
\end{theorem}

The purpose of defining inherited properties is just that we can compare our property on a whole Latin square to a property on an initial segment of the triangle removal process (a direct comparison does not make sense since the triangle removal process is unlikely to complete a full Latin square).

Next, instead of analyzing the triangle removal process directly it is more convenient to compare an initial fraction of it to an independent model. Recall the definition of $\mbf B_{n,p}$ from \cref{def:stable}.

\begin{lemma}[{\cite[Lemma~5.2]{KSS21}}]\label{lem:TRP-independent}
Let $\mc{T}$ be a property of unordered partial Latin squares that is monotone decreasing, i.e., if $P\in\mc{T}$ and $P'\subseteq P$ then $P'\in\mc{T}$. Fix $\alpha\in(0,1)$, let $\mbf{P}\sim\mb{L}(n,\alpha n^2)$, and let $\mbf{G}^\ast$ be the partial Latin square obtained from $\mbf{B}_{n,\alpha/n}$ by simultaneously deleting every hyperedge which intersects another hyperedge in more than one vertex. Then 
\[\Pr[\mbf{P}\in\mc{T}]\le O(\Pr[\mbf{G}^{\ast}\in\mc{T}]).\]
\end{lemma}

To apply this machinery we must first show that the property of having few nondegenerate cuboctahedra satisfies the inheritance property defined in \cref{def:inheritance}.
\begin{lemma}\label{lem:octahedra-inheritance}
Fix $\alpha,\delta\in(0,1)$. Let $\mc{T}^\delta\subseteq\mc{L}$ be the property that a Latin square $L\in\mc{L}$ has at most $(1-\delta)n^4$ nondegenerate cuboctahedra, let $m=\alpha n^2$, and let $\mc{T}_m^\delta\subseteq\mc{L}_m$ be the property that a partial Latin square $P\in\mc{L}_m$ has at most $\alpha^8(1-\delta/2)n^4$ nondegenerate cuboctahedra. Then $\mc{T}_m^\delta$ is $(1/2)$-inherited from $\mc{T}^\delta$.
\end{lemma}
\begin{proof}
Fix $L\in\mc{T}^\delta$. Let $\mbf{L}_m$ be obtained by taking $m$ uniformly random hyperedges of $L$, and let $\mc{Q}$ be the set of nondegenerate cuboctahedra in $L$. For $Q\in\mc{Q}$, let $\mbf{1}_Q$ be the indicator variable for the event that $Q\subseteq\mbf{L}_m$, and let $\mbf{X}=\sum_{Q\in\mc{Q}}\mbf{1}_Q$ be the number of nondegenerate cuboctahedra in $\mbf{L}_m$.
For all $Q\in\mc{Q}$ we have $\mb{E}\mbf{1}_Q=\alpha^8+O(1/n)$, so $\mb{E}\mbf{X}\le\alpha^8(1-\delta+o(1))n^4$. Also, for each pair of disjoint $Q,Q'\in\mc{Q}$ we have $\on{Cov}(\mbf{1}_Q,\mbf{1}_{Q'})=O(1/n)$.
In every Latin square, every nondegenerate cuboctahedron shares a hyperedge with at most $8n^3$ others (choose which hyperedge overlaps, then choose the identities of one vertex adjacent to each of those three vertices in the cuboctahedron; repeatedly applying the Latin property, we see that there is at most one choice for the remaining vertices). Since $L\in\mc{T}^\delta$ has $|\mc{Q}|\le n^4$, there are $O(n^7)$ intersecting pairs of nondegenerate cuboctahedra in $\mc{Q}$.
Thus $\on{Var}\mathbf X\le|\mc{Q}|^2\cdot O(1/n) + O(n^7) = O(n^7)$. By Chebyshev's inequality, we conclude that
\[\Pr[\mbf{L}_m\in\mc{T}_m^\delta]=\Pr[\mbf{X}\le\alpha^8(1-\delta/2)n^4]=1-o(1)>1/2.\]
That is to say, $\mc{T}_m^\delta$ is $(1/2)$-inherited from $\mc{T}^\delta$.
\end{proof}

Finally, we will need the following concentration inequality. The statement presented here appears for example in \cite[Theorem~2.11]{Kwa20}, and follows from an inequality of Freedman~\cite{Fre75}.
\begin{theorem}\label{thm:freedman-concentration}
Let $\bs\omega=(\bs\omega_1,\ldots,\bs\omega_N)$ be a sequence of independent, identically distributed random variables with $\Pr[\bs\omega_i=1]=p$ and $\Pr[\bs\omega_i=0]=1-p$. Let $f:\{0,1\}^N\to\mb{R}$ satisfy $|f(\bs\omega)-f(\bs\omega')|\le K$ for all pairs $\bs\omega,\bs\omega'\in\{0,1\}^N$ differing in exactly one coordinate. Then
\[\Pr[|f(\bs\omega)-\mb{E}f(\bs\omega)|>t]\le \exp\left(-\frac{t^2}{4K^2Np+2Kt}\right).\]
\end{theorem}

We are ready to prove \cref{thm:strong-lower-cuboctahedra}.
\begin{proof}[Proof of \cref{thm:strong-lower-cuboctahedra}]
Let $\alpha > 0$, which will later be chosen to be small with respect to $\delta$. Let $\mbf{G}^\ast$ be as in \cref{lem:TRP-independent}. Let $\mc{T}_m^\delta$ be the property that a partial Latin square, not necessarily having exactly $m$ edges, has at most $(1-\delta/2)\alpha^8n^4$ nondegenerate cuboctahedra. This property is clearly monotone decreasing. Let $\mbf{P}\sim\mb{L}(n,\alpha n^2)$ and $\mbf{L}$ be a uniformly random order-$n$ Latin square. Let $X_H(\mbf{L})$ be the number of nondegenerate cuboctahedra in $\mbf{L}$. Then \cref{lem:octahedra-inheritance,thm:transference-TRP} along with \cref{lem:TRP-independent} show
\[\Pr[X_H(\mbf{L})\le(1-\delta)n^4]\le\exp(n^{2-\gamma})\Pr[\mbf{P}\in\mc{T}_m^\delta]\le\exp(O(n^{2-\gamma}))\cdot O(\Pr[\mbf{G}^\ast\in\mc{T}_m^\delta]).\]

It now suffices to study $\mbf{G}^\ast$ (via $\mbf B_{n,\alpha/n}$, which determines it). Let $\mbf{Q}'$ be the maximum size of a collection of nondegenerate cuboctahedra in $\mbf{G}^\ast$ for which every hyperedge is in at most $2\alpha^7n^2$ of the cuboctahedra in the collection. We claim that
\begin{equation}\label{eq:collection-mean}
\mb{E}\mbf{Q}'\ge\alpha^8(1-\delta/4)n^4,
\end{equation}
if $\alpha$ is sufficiently small with respect to $\delta$ (and $n$ is sufficiently large). For the moment, assume this is true.

Note that we can view $\mbf{Q}'$ as a function of $n^3$ different independent zero-one random variables (one for each possible hyperedge of $\mbf B_{n,\alpha/n}$). If we add a hyperedge $e$ to $\mbf B_{n,\alpha/n}$, this can result in at most one hyperedge being added to $\mbf{G}^\ast$ ($e$ itself) and it can result in at most 3 hyperedges being removed from  $\mbf{G}^\ast$ (hyperedges which share more than one vertex with $e$). Similarly, removing a hyperedge from $\mbf B_{n,\alpha/n}$ affects $\mbf{G}^\ast$ by at most 3 hyperedges. Now, adding a hyperedge to $\mbf{G}^\ast$ increases $\mbf{Q}'$ by at most $2\alpha^7 n^2$ (and can never decrease $\mbf{Q}'$). So, $\mbf{Q}'$ is a $O(n^2)$-Lipschitz. \cref{thm:freedman-concentration} shows that
\begin{align*}
\Pr[\mbf{G}^\ast\in\mc{T}_m^\delta]&\le\Pr[\mbf{Q}'\le\alpha^8(1-\delta/2)n^4]\le\Pr[\mbf{Q}'\le\mb{E}\mbf{Q}'-\alpha^8\delta n^4/4]\\
&\le\exp\bigg(-\frac{(\alpha^8\delta n^4/4)^2}{4(O(n^2))^2n^3(\alpha/n)+2(O(n^2))(\alpha^8\delta n^4/4)}\bigg) = \exp(-\Omega_{\alpha,\delta}(n^2)),
\end{align*}
and the result follows.

Now it suffices to prove \cref{eq:collection-mean}. Let $\mbf{Q} = X_H(\mbf{G}^\ast)$ be the number of nondegenerate cuboctahedra in $\mbf{G}^\ast$. For a triple $T$ let $\mbf{Q}_T$ be the number of nondegenerate cuboctahedra in $\mbf{G}\cup\{T\}$ which include the hyperedge $T$, and let $\mbf{Q}_2$ be the sum of $\mbf Q_T$ over all $T\in \mbf B_{n,\alpha/n}$ for which $\mbf Q_T\ge 2\alpha^7 n^2$. We have
\[\mbf{Q}'\ge\mbf{Q}-\mbf{Q}_2\]
since we can consider the collection of nondegenerate cuboctahedra in $\mbf{G}^\ast$ and simply remove all cuboctahedra which have an edge that is in more than $2\alpha^7n^2$ cuboctahedra of $\mbf{B}_{n,\alpha/n}$. Now,
\[\mb{E}\mbf{Q} = (1+O(1/n))n^{12}\cdot\bigg(\frac{\alpha}{n}\bigg)^8\bigg(1-\frac{\alpha}{n}\bigg)^{24n+O(1)} = (e^{-24\alpha}+O(1/n))\alpha^8n^4\]
by linearity of expectation: a specific nondegenerate cuboctahedron will be in $\mbf{G}^\ast$ precisely if all its 8 hyperedges are in $\mbf{G}$ and all triples sharing an edge with one of these 8 hyperedges (of which there are $24n+O(1)$) are not in $\mbf{G}$.

We now fix a triple $T$ and study $\mbf{Q}_T$, which is a degree-7 polynomial of independent random variables. We see $\mu = \mb{E}\mbf{Q}_T = (1+O(1/n))n^9(\alpha/n)^7 = \alpha^7n^2 + O(n)$, and a straightforward second-moment calculation yields $\on{Var}\mbf{Q}_T = O(n^3)$. Therefore
\[\Pr[\mbf{Q}_T\ge 2\alpha^7n^2]\lesssim_\alpha n^{-1}.\]
In particular, for sufficiently large $n$ we have
\begin{align*}
\mb{E}\mbf{Q}_2 &= n^3\bigg(\frac{\alpha}{n}\bigg)\mb{E}\mbf{Q}_T\mbm{1}_{\mbf{Q}_T\ge 2\alpha^7n^2}\le 3\alpha n^2\mb{E}(\mbf{Q}_T-\mu)\mbm{1}_{\mbf{Q}_T\ge 2\alpha^7n^2}\\
&\le 3\alpha n^2(\mb{E}(\mbf{Q}_T-\mu)^2)^{1/2}\Pr[\mbf{Q}_T\ge 2\alpha^7n^2]^{1/2}\le 3\alpha n^2(O(n^3))^{1/2}(O_\alpha(n^{-1}))^{1/2}\lesssim_\alpha n^3.
\end{align*}
Thus as long as $e^{-24\alpha}\ge 1-\delta/8$ we find
\[\mb{E}\mbf{Q}'\ge\mb{E}\mbf{Q}-\mb{E}\mbf{Q}_2\ge(1-\delta/4)\alpha^8n^4\]
for $n$ sufficiently large. This establishes \cref{eq:collection-mean} and we are done.
\end{proof}
\begin{remark}
It appears that the method of proof of \cref{thm:strong-lower-cuboctahedra} (and similarly \cite[Theorem~1.2(a)]{KSS21}) may apply to more general colored triple systems $H$, though we do not pursue this here.
\end{remark}

\section{High Girth Latin Squares}\label{sec:high-girth}
Recall that a \textit{Steiner triple system} of order $N$ is an $N$-vertex triple system (i.e., $3$-uniform hypergraph) such that every pair of vertices is contained in exactly one triple. Equivalently, this is a \textit{triangle-decomposition} of the complete graph $K_N$. A foundational theorem of Kirkman \cite{kirkman1847problem} states that order-$N$ Steiner triple systems exist if and only if $N \equiv 1,3 \pmod{6}$. The necessity of the arithmetic condition is easy to see: a graph can have a triangle-decomposition only if all its degrees are even and the total number of edges is a multiple of three.

As mentioned in the introduction, in \cite{KSSS22} the authors proved the analogue of \cref{thm:high-girth-latin-square} for Steiner triple systems. We begin by recalling the relevant definition and theorem.

\begin{definition}\label{def:triple system girth}
	The \textit{girth} of a triple system $\mc S$ is the smallest $g>3$ such that there exists a set of $g$ vertices spanning $g-2$ triangles of $\mc S$. If $\mc S$ contains no such vertex set, we say it has infinite girth.
\end{definition}

\begin{theorem}[\text{\cite[Theorem 1.1]{KSSS22}}]\label{thm:high-girth-sts}
	Given $g\in\mb{N}$, there is $N_g\in\mb{N}$ such that if $N \ge N_g$ and $N$ is congruent to $1$ or $3\pmod{6}$, then there exists a Steiner triple system of order $N$ and with girth greater than $g$.
\end{theorem}

Order-$N$ Latin squares are naturally equivalent to triangle-decompositions of the complete tripartite graph $K_{N,N,N}$ (the three parts correspond to rows, columns, and symbols). In this language, the definition of girth in the introduction is equivalent to \cref{def:triple system girth} and \cref{thm:high-girth-latin-square} is equivalent to the statement that for every fixed $g$ and sufficiently large $N \in \mb N$ there exists a triangle-decomposition of $K_{N,N,N}$ with girth greater than $g$.

In this section we will outline the proof of \cite[Theorem 1.1]{KSSS22} and explain the adaptations necessary to prove \cref{thm:high-girth-latin-square}. We omit some details (even salient ones!) that are the same in the tripartite and non-partite cases. For a more detailed outline, addressing some of the subtleties and difficulties involved, we refer the reader to \cite[Section 2]{KSSS22} (the full proof of \cref{thm:high-girth-sts} is of course contained in \cite{KSSS22} as well).

We will need the following notation and definition.

\begin{notation}
	For a graph $G\subseteq K_{N,N,N}$, we write $G_{i,j}=G[V^i\cup V^j]$ for the graph of edges between $V^i$ and $V^j$. Throughout, indices that naturally come in threes due to a tripartite structure are taken modulo $3$ (so, for example, every edge lies in $G_{j-1,j}$ for some $j\in \{1,2,3\}$).
\end{notation}

\begin{definition}
	We say that $G \subseteq K_N$ is \textit{triangle-divisible} if all the vertex degrees in $G$ are even and $|E(G)|$ is a multiple of three. We call $G \subseteq K_{N,N,N}$ \textit{triangle divisible} if for every $1 \leq j \leq 3$ and $v \in V^j$, we have $\deg_G(v,V^{j-1}) = \deg_G(v,V^{j+1})$.
\end{definition}

We remark that the previous definition contains a certain abuse: $K_{N,N,N}$ can be viewed as a subgraph of $K_M$, for $M \geq 3N$. However, it will always be clear from context whether we are thinking of $G$ as a tripartite graph.

Note that if $G$ is obtained from a triangle divisible graph $H$ by removing a triangle-divisible subgraph of $H$ (so, in particular, by removing a collection of edge-disjoint triangles) then $G$ itself is triangle divisible.

\subsection{High-girth iterative absorption in a nutshell}\label{sub:nutshell}

Let $g \in \mb N$ and let $K$ be either $K_N$ or $K_{N,N,N}$. If $K=K_N$ we additionally assume that $N \equiv 1,3 \pmod{6}$. Denote by $V^1,V^2,V^3$ the three parts of $K_{N,N,N}$. We wish to prove that if $N$ is sufficiently large then there exists a triangle-decomposition of $K$ with girth greater than $g$. We do so by describing a probabilistic algorithm that whp constructs such a decomposition. We build on the method of \textit{iterative absorption}, and especially its application to triangle-decompositions described in \cite{BGKLMO20}.

Besides iterative absorption, the other major ingredient in the proof is the \textit{high girth triangle removal process}. This is a simple random greedy algorithm for constructing partial triangle-decompositions with girth greater than $g$, defined as follows.  Beginning with an empty collection of triangles $\mc C(0)$, for as long as possible, choose a triangle $T^*$ in $K$ uniformly at random, subject to the constraint that $\mc C(t) \cup \{T^*\}$ is a partial triangle-decomposition of $K$ with girth greater than $g$. Then, set $\mc C(t+1) = \mc C(t) \cup \{T^*\}$. This process, with $K=K_N$, was analyzed independently by Glock, K\"uhn, Lo, and Osthus \cite{GKLO20} and by Bohman and Warnke \cite{BW19}. They showed that whp the process constructs an approximate triangle-decomposition of $K$ (i.e., covering all but a $o(1)$-fraction of $E(K)$). A generalization of this process \cite[Section 9]{KSSS22} was used by the authors as a crucial component in the proof of \cref{thm:high-girth-sts}. In particular, the more general process is applicable to the tripartite setting of \cref{thm:high-girth-latin-square}, and allows a more general family of constraints than just those corresponding to girth.

The high-girth iterative absorption procedure can be broken down as follows:

\subsubsection*{The vortex}

We fix a small constant $\rho>0$, and a \textit{vortex} $V(K) = U_0 \supseteq U_1 \supseteq \ldots \supseteq U_\ell$, where $\ell = O(1)$ is a large constant and $|U_{i+1}| = (1+o(1))|U_i|^{1-\rho}$ for each $i<\ell$. We note that $|V(K)| = (1+o(1))|U_\ell|^{(1-\rho)^{-\ell}} = |U_\ell|^{O(1)}$. If $K=K_{N,N,N}$, we additionally require that for every $i$, the intersection of $U_i$ with each of $V^1,V^2$, and $V^3$ has the same cardinality.

\subsubsection*{The absorber graph}

We set aside an \textit{absorber} $H \subseteq K$. This is a graph, containing $U_\ell$ as an independent set, with the property that for every triangle-divisible graph $L \subseteq K[U_\ell]$, the graph $H \cup L$ admits a triangle-decomposition with girth greater than $g$. The construction for $K=K_N$ is given in \cite[Theorem 4.1]{KSSS22}. However, this construction is not tripartite, as would be required when $K=K_{N,N,N}$. \cref{thm:absorber}, which we prove below, is the necessary tripartite analogue of \cite[Theorem 4.1]{KSSS22}.

\subsubsection*{Initial sparsification}

After setting aside the absorber, we perform the high-girth triangle removal process to obtain a partial Steiner triple system $\mc I$ covering all but a $o(1)$-fraction of the edges in $K \setminus H$. We obtain the lower bounds in \cite[Theorem 1.3]{KSSS22} and \cref{thm:pasch-free} by essentially multiplying the number of choices at each step of this process (we detail the calculation necessary to obtain \cref{thm:pasch-free} in \cref{sub:counting-intercalate-free}). We remark that the initial sparsification has an important role beyond its usefulness in enumeration; see \cite{KSSS22} for more details.

\subsubsection*{Cover down}

The heart of the iterative absorption machinery is a ``cover down'' procedure. At each step $k$, given a set of edge-disjoint triangles  $\mc M_0\cup\dots\cup \mc M_{k-1}$ in $K\setminus H$ covering all possible edges that are not contained in $U_k$, we find an augmenting set of edge-disjoint triangles $\mc M_k$ which covers all possible edges except ones contained in $U_{k+1}$. We do this in such a way that $\mc M_0\cup\dots\cup \mc M_{k}$ has girth greater than $g$.

After $\ell$ steps of this procedure, we will have found a set of edge-disjoint triangles $\mc M_0\cup\dots\cup \mc M_{k-1}$ in $E(K) \setminus E(H)$ covering all edges that are not contained in $U_\ell$. We then use the defining property of the absorber to transform this into a triangle-decomposition of $K$. (Technically, beyond the fact that $\mc M_0\cup\dots\cup \mc M_{k-1}$ has high girth, we need to maintain the property that $\mc M_0\cup\dots\cup \mc M_{k-1}\cup \mc B$ has high girth, for each of the possible triangle-decompositions $\mc B$ associated with the absorber $H$).

Each iteration of the cover down procedure is proved using a three-stage randomized algorithm. The task is to find a set of edge-disjoint triangles in a particular graph $G$, covering all edges except those in a particular set $U$, and avoiding ``obstructions to having high girth'' with a particular collection of previously selected triangles.
\begin{itemize}
	\item First, we run a generalized version of the high-girth triangle removal process to find an appropriate set of triangles $\mc M^*$ covering \emph{almost} all of the edges which do not lie in $U$ (and none of the edges inside $U$). Most of the leftover edges will be contained in a quasirandom \emph{reserve graph} $R$, which is set aside before starting the process (this is a convenient way to control the approximate structure of the graph of leftover edges).
	
	In order for the process to succeed, it is necessary for its initial conditions to be quite regular. Specifically, for every edge, the number of available triangles including that edge (i.e., those that we permit ourselves to use) must vary by at most a multiplicative factor of $1 \pm |U|^{-c}$, with $c$ an absolute constant.
	
	Na\"ively, one might think to simply take our set of available triangles to be the set of triangles whose addition would not violate our girth condition. However, this set of triangles is not regular enough, and it is necessary to perform a \textit{regularity boosting} step to obtain an appropriately regular subset of these triangles. Specifically, in \cite[Lemma 5.1]{KSSS22} (which is based on \cite[Lemma 4.2]{BGKLMO20}), we prove that given a set $\mc T$ of triangles satisfying certain weak regularity and ``extendability'' conditions, one can find a subset $\mc T'\subseteq T$ satisfying much stronger regularity conditions. This is accomplished by fixing an appropriate weight for each triangle in $\mc T$, and randomly subsampling with these weights as probabilities. The weights (which essentially correspond to a \emph{fractional} triangle-decomposition of an appropriate graph) are constructed using weight-shifting ``gadgets'' defined in terms of copies of the complete graph $K_5$. Suitable gadgets are not available in the tripartite setting. Hence, we build on the more sophisticated techniques of Bowditch and Dukes \cite{BD19} and Montgomery \cite{Mon17} to prove \cref{lem:reg}, which is the necessary tripartite regularity-boosting lemma. As stated in the introduction, this regularity-boosting lemma is the most substantial difference between the partite and non-partite cases.
	
	\item The remaining leftover edges (i.e., the edges in $G\setminus G[U]$ which are not covered by $\mc M^*$) can be classified into two types: ``internal'' edges which lie completely outside $U$, and ``crossing'' edges which have a single vertex in $U$. To handle the remaining internal edges, we use a random greedy algorithm to choose, for each leftover internal edge $e$, a suitable covering triangle $T_e$.
	
	\item At this point, the only leftover edges each have one vertex outside $U$ and one vertex inside $U$. This allows us to reduce the problem to a simultaneous perfect matching problem: For every $v \in V(G) \setminus U$, let $W_v \subseteq U$ be the set of vertices $u$ such that $uv$ is still uncovered. Suppose that $M_v$ is a perfect matching of $W_v$. Then $\{ v \cup e : e \in M_v \}$ is a set of edge-disjoint triangles covering all the leftover crossing edges incident to $v$. Thus, it suffices to find a family of perfect matchings $\{M_v\}_{v \in V(G) \setminus U}$ such that the corresponding sets of triangles are edge-disjoint and satisfy appropriate girth properties.
	
	In \cite{KSSS22} (i.e., the non-partite case $K=K_N$), the matchings $M_v$ are found using a robust version of Hall's matching condition (cf.\ \cite[Section 6]{KSSS22}), applied to a uniformly random balanced bipartition of $U$. In the tripartite case we consider here the situation is simpler: the graph induced by each $W_v$ is already bipartite due to the structure of $K_{N,N,N}$. Moreover, this bipartition of $W_v$ is balanced, as $W_v$ is obtained by removing a triangle-divisible graph from $K_{N,N,N}$. In particular we note that although the majority of the link graph is comprised of edges from the reserve graph, one does not need to maintain ``divisibility'' constraints for the reserve graph and instead divisibility follows simply by construction.
\end{itemize}

The final difference between the tripartite and non-partite cases concerns conditions regarding the graph of uncovered edges and the set of available triangles that must be verified before each stage of the iteration. One such set of conditions, called ``iteration-typicality'', is given in \cite[Definition 10.1]{KSSS22}. These must be replaced by a tripartite analogue, which we now define.

\begin{definition}[Tripartite iteration-typicality]\label{def:tripartite-iteration-typical}
	Let $n \leq N$. Fix a descending sequence of subsets $V(K_{n,n,n}) = U_k\supseteq \dots \supseteq U_\ell$. Consider a graph $G\subseteq K_{n,n,n}$ and a set of triangles $\mc{A}$ in $G$. We say that $(G,\mc{A})$ is \emph{$(p,q,\xi,h)$-tripartite-iteration-typical} (with respect to our sequence of subsets) if for every $1 \leq j \leq 3$:
	\begin{itemize}
		\item for every $k \le i < \ell$, every set $W \subseteq (V^j \cup V^{j+1}) \cap U_i$ of at most $h$ vertices is adjacent (with respect to $G$) to a $(1\pm \xi)p^{|W|}$ fraction of the vertices in $U_i \cap V^{j-1}$ and $U_{i+1} \cap V^{j-1}$; and
		
		\item for any $i,i^*$ with $k\le i <\ell$, and $i^*\in \{i,i+1\}$, and any edge subset $Q\subseteq G[U_i] \cap G_{j,j+1}$ spanning $|V(Q)|\le h$ vertices, a $(1\pm\xi)p^{|V(Q)|}q^{|Q|}$-fraction of the vertices $u\in U_{i^*} \cap V^{j-1}$ are such that $uvw\in\mc{A}$ for all $vw\in Q$.
	\end{itemize} 
\end{definition}

To summarize, here are the changes required to the proof of \cref{thm:high-girth-sts} in order to prove \cref{thm:high-girth-latin-square}:

\begin{itemize}
	\item The vortex $V(K_{N,N,N})=U_0\supseteq \dots\supseteq U_\ell$ should be chosen such that each $U_k$ has the same number of vertices in each $V^j$.
	
	\item In the final step of the cover down procedure, in the non-partite case the problem was first reduced to a bipartite matching problem by taking a random bipartition of $U_{i+1}$. In the tripartite setting this is not necessary as the bipartite structure is already induced by $K_{N,N,N}$.
	
	\item During the iterations, replace iteration-typicality (\cite[Definition 10.1]{KSSS22}) with tripartite iteration-typicality (\cref{def:tripartite-iteration-typical}) with $h=6$. The analysis in \cite[Section~9.1]{KSSS22}, where iteration-typicality is first established, requires only very minimal changes. Similarly, the verification that iteration-typicality is maintained throughout the iterations (which is done in \cite[Section 10.4.2]{KSSS22}) requires only small changes.
	
	\item In \cite{KSSS22}, the regularity boosting step that precedes the high-girth triangle removal process relies on ``gadgets'' which are not tripartite. We prove a tripartite regularity boosting lemma in \cref{sub:triangle-regularization}.
	
	\item The absorbing structure must be tripartite. We give such a construction in \cref{sec:absorbers}.
\end{itemize}

The remainder of this section is devoted to constructing tripartite high-girth absorbers (\cref{sec:absorbers}), proving a tripartite regularity boosting lemma (\cref{sub:triangle-regularization}), and providing the calculations that yield the lower bound in \cref{thm:pasch-free} (\cref{sub:counting-intercalate-free}).

\subsection{Efficient tripartite high-girth absorbers}\label{sec:absorbers}
Recall that a tripartite graph is \emph{triangle-divisible} if for every vertex, its degree to both other parts is the same. This is a necessary but insufficient condition for triangle-decomposability. In this section we explicitly define a high-girth ``absorbing structure'' that will allow us to find a triangle-decomposition extending any triangle-divisible tripartite graph on a specific triple of vertex sets. Importantly, the size of this structure is only polynomial in the size of the distinguished vertex sets, which is needed for our proof strategy. For a set of triangles $\mc R$, let $V(\mc R)$ be the set of all vertices in these triangles.

The following theorem encapsulates the properties of our absorbing structure, and is basically the same as \cite[Theorem~4.1]{KSSS22} (we just need to be a bit careful to ensure that everything is tripartite). Throughout, every tripartite graph will have a \emph{fixed} tripartition (i.e., each vertex has a color in $\{1,2,3\}$). When we refer to a subgraph of a graph $G$, this subgraph inherits the tripartition $G$.

\begin{theorem}\label{thm:absorber}
There is $C_{\ref{thm:absorber}}\in\mb{N}$ so that for $g\in\mb{N}$ there exists $M_{\ref{thm:absorber}}(g)\in\mb{N}$ such that the following holds. For any $m\ge 1$, there is a tripartite graph $H$ with at most $M_{\ref{thm:absorber}}(g)m^{C_{\ref{thm:absorber}}}$ vertices containing a distinguished independent set $X^1\cup X^2\cup X^3$ (where each $X^j$ contains exactly $m$ vertices of color $j$), satisfying the following properties.
\begin{enumerate}[{\bfseries{Ab\arabic{enumi}}}]
    \item\label{AB1} For any triangle-divisible tripartite graph $L$ on $X := X^1\cup X^2\cup X^3$ (with a consistent tripartition) there exists a triangle-decomposition $\mc S_L$ of $L\cup H$ which has girth greater than $g$.
    
    \item\label{AB2} Let $\mc{B} = \bigcup_L\mc{S}_L$ (where the union is over all triangle-divisible tripartite graphs $L$ on the vertex set $X$) and consider any tripartite graph $K$ containing $H$ as a subgraph. Say that a triangle in $K$ is \emph{nontrivially $H$-intersecting} if it is not one of the triangles in $\mc B$, but contains a vertex in $V(H)\setminus X$.
    
    Then, for every set of at most $g$ triangles $\mc{R}$ in $K$, there is a subset $\mc{L}_{\mc{R}}\subseteq\mc{B}$ of at most $M_{\ref{thm:absorber}}(g)$ triangles such that every Erd\H{o}s configuration $\mc{E}$ on at most $g$ vertices which includes $\mc{R}$ must either satisfy $\mc{E}\cap\mc{B}\subseteq\mc L_\mc R$ or must contain a nontrivially $H$-intersecting triangle $T\notin\mc R$.
\end{enumerate}
\end{theorem}

\begin{remark}
Note that \cref{AB1}, with $L$ as the empty graph on the vertex set $X$, implies that $H$ itself is triangle-decomposable (hence triangle-divisible).
\end{remark}

We will prove \cref{thm:absorber} by chaining together some special-purpose graph operations. Call a cycle in a tripartite graph a \emph{tripartite cycle} if its length is divisible by $3$ and every third vertex is in the same part.

\begin{definition}[Path-cover]\label{def:path-cover}
Let the \emph{path-cover} $\wedge X$ of a vertex set $X = X^1\sqcup X^2\sqcup X^3$ be the graph obtained as follows. Start with the empty graph on $X$. Then, for every unordered pair $u,v$ of distinct vertices in each part $X^j$, add $12|X|^2$ new paths of length $3$ between $u$ and $v$, introducing $2$ new vertices for each (so in total, we introduce $\sum_j 24|X|^2\binom{|X^j|}{2}$ new vertices). We call these new length-$3$ cycles \emph{augmenting paths}. Note that there are two ways to properly color a length-$3$ path between $u$ and $v$; we include exactly $6|X|^2$ of each type, so the augmenting paths between $u$ and $v$ can be decomposed into $6|X|^2$ tripartite 6-cycles.
\end{definition}

The key point is that for any triangle-divisible graph $L$ on the vertex set $X$, the edges of $L\cup\wedge X$ can be decomposed into short cycles.

\begin{lemma}\label{lem:path-cover}
If a tripartite graph $L$ on $X$ is triangle-divisible, then $L\cup\wedge X$ can be decomposed into tripartite cycles of length at most $9$. Additionally, if $L$ is triangle-divisible, then so is $L\cup\wedge X$.
\end{lemma}

The proof is analogous to that of \cite[Lemma~4.3]{KSSS22}. First, we note that triangle-divisible graphs have a decomposition into tripartite cycles (start at a vertex and move around the graph cyclically until a collision occurs, then remove this cycle and continue). We then ``shorten'' these cycles with our augmenting paths.

Next, it is a consequence of \cite[Lemma~6.5]{BKLOT17} that for any triangle-divisible tripartite graph $H$, there is a tripartite graph $A(H)$ containing the vertex set of $H$ as an independent set such that $A(H)$ and $A(H)\cup H$ are both triangle-decomposable. (In our case, we only care about the case $H\in \{C_3,C_6,C_9\}$, in which case it is not hard to construct $A(H)$ by hand.)

\begin{definition}[Cycle-cover]\label{def:cycle-cover}
Let the \emph{cycle-cover} of a tripartite vertex set $Y$ be the graph $\triangle Y$ obtained as follows. Beginning with $Y$, for every $H\in\{C_3,C_6,C_9\}$ and every color-preserving injection $f:V(H)\to Y$ we add a copy of $A(H)$, such that each $v\in V(H)\subseteq V(A(H))$ in the copy of $A(H)$ coincides with $f(v)$ in $Y$. We do this in a vertex-disjoint way, introducing $|V(A(H))\setminus V(H)|$ new vertices each time. (Think of the copy of $A(H)$ as being ``rooted'' on a specific set of vertices in $Y$, and otherwise being disjoint from everything else.)
\end{definition}

\begin{lemma}\label{lem:cycle-cover}
Let $Y = V(\wedge X)$ be the vertex set of the graph $\wedge X$. If a graph $L$ on $X$ is triangle-divisible, then $L\cup\wedge X\cup\triangle Y$ admits a triangle-decomposition.
\end{lemma}

The proof of \cref{lem:cycle-cover} is the same as that of \cite[Lemma~4.5]{KSSS22}, so we omit the details.

If we consider any $X = X_1\sqcup X_2\sqcup X_3$ with $|X_i| = m$, then \cref{lem:cycle-cover} implies that $\wedge X\cup\triangle(V(\wedge X))$ can ``absorb'' any triangle-divisible tripartite graph on $X$, though not necessarily in a high-girth manner. We use a tripartite modification of the \emph{$g$-sphere-cover} of \cite[Definition~4.6]{KSSS22} to provide a girth guarantee (the definition is basically the same, but we glue a $g$-sphere only onto those triples which are tripartite).

\begin{definition}[Sphere-cover]\label{def:sphere-cover}
Let the \emph{$g$-sphere-cover} of a tripartite vertex set $Z$ be the graph $\medcircle_gZ$ obtained by the following procedure. For every tripartite triple $T$ of vertices of $Z$, arbitrarily label these vertices as $a,b_1,b_2$. Then, append a ``$g$-sphere'' to the triple. Namely, first add $2g-1$ new vertices $b_3,\ldots,b_{2g},c$. Then add the edges $ab_j$ for $3\le j\le 2g$, the edges $cb_j$ for $1\le j\le 2g$, the edges $b_jb_{j+1}$ for $2\le j\le 2g-1$, and the edge $b_{2g}b_1$. This can be seen to preserve the tripartite property.

Note that every such $g$-sphere $Q$ itself has a triangle-decomposition: specifically, we define the \emph{out-decomposition} to consist of the triangles
\[c b_2 b_3,\,ab_3 b_4,\, c b_4b_5,\, a b_5 b_6,\dots,cb_{2g}b_1.\]
We also identify a particular triangle-decomposition of the edges $Q\cup T$: the \emph{in-decomposition} consists of the triangles 
\[cb_1b_2,\,ab_2b_3,\,cb_3b_4,\,a b_4 b_5,\dots,ab_{2g}b_1.\]
For a triple $T$ in $Z$, let $\mc{B}^\medcircle(T)$ be the set of all triangles in the in- and out-decompositions of the $g$-sphere associated to $T$. We emphasize that $T\notin\mc{B}^\medcircle(T)$.
\end{definition}

Finally, the proof of \cref{thm:absorber} is exactly like the proof of \cite[Theorem~4.1]{KSSS22}: let $Y = V(\wedge X)$, $Z = V(\triangle Y)$, and take $H = \wedge X\cup\triangle Y\cup\medcircle_gZ$. The girth properties guaranteed by analogues of \cite[Lemmas~4.7,~4.8]{KSSS22} are enough to show the desired girth properties here.

\subsection{Triangle-regularization of tripartite graphs}\label{sub:triangle-regularization}
Given a tripartite set of triangles with suitable regularity and ``extendability'' properties, the following lemma finds a subset which is substantially more regular. This lemma can be used in place of \cite[Lemma~5.1]{KSSS22} (which itself is an adaptation of \cite[Lemma~4.2]{BGKLMO20}) for the proof of \cref{thm:high-girth-latin-square}.
Unlike in \cite[Lemma~5.1]{KSSS22}, some notion of ``approximate triangle-divisibility'' is needed, since every triangle has exactly one edge between each pair of parts.%, a property which is not changed in the setting of fractional matchings.

\begin{lemma}\label{lem:reg}
There are $C_{\ref{lem:reg}}\in\mb{N}$ and $n_{\ref{lem:reg}}:\mb{N}\times\mb{R}\to\mb{N}$ such that the following holds. We are given $q\in(0,1)$ and $C\ge C_{\ref{lem:reg}}$, and let $n\ge n_{\ref{lem:reg}}(C,q)$. Suppose $q\in(0,1)$, $\xi = C^{-8}$, and $p\in (n^{-1/12},1)$, let $G$ be a balanced tripartite graph on $3n$ vertices $V_1\sqcup V_2\sqcup V_3$, and let $\mc{T}$ be a collection of triangles of $G$, satisfying the following properties.
\begin{enumerate}
    \item Every edge $e\in E(G)$ is in $(1\pm\xi)p^2qn$ triangles of $\mc{T}$.
    
    \item\label{item:graph-extension-bounds} For every $j\in[3]$ and every set $S\subseteq V^{j-1}\cup V^j$ with $|S|\le 6$ in $G$, there are $(1\pm\xi)p^{|S|}n$ common neighbors of $S$ in $V_{j+1}$.
    
    \item\label{item:triangle-extension-bounds} For every $j\in[3]$ and every set $Q$ of at most 6 edges between $V^{j-1}$ and $V^{j}$ with $|Q|\le 6$, there are between $C^{-1}p^{|V(E)|}n$ and $Cp^{|V(Q)|}n$ vertices $u\in V_{j+1}$ which form a triangle in $\mc{T}$ with every edge in $Q$.
    
    \item\label{item:nearly-divisible} $G$ has the same number of edges between each pair of parts, and for every $j\in[3]$ and every $v\in V^j$, we have $|\deg_{V^{j-1}}(v)-\deg_{V^{j+1}}(v)|\le n^{2/3}$.
\end{enumerate}
Then, there is a subcollection $\mc{T}'\subseteq\mc{T}$ such that every edge $e\in E(G)$ is in $(1\pm n^{-1/4})p^2qn/4$ triangles of $\mc{T}'$.
\end{lemma}
\begin{remark}
In order to apply the cover down scheme outlined in \cref{sub:nutshell}, the reserve graph $R$ will need to be chosen in such a way that the fourth item above is satisfied. As in \cite{KSSS22}, we sample $R$ randomly, but instead of independently including each edge with a particular probability, we choose a uniformly random subset, of a particular size, of the edges between each part.
%to be a random subset by taking random samples and then conditioning on edge-counts.
\end{remark}
As in \cite[Lemma~5.1]{KSSS22} and \cite[Lemma~4.2]{BGKLMO20}, we will prove \cref{lem:reg} by sampling from a fractional clique-decomposition. However, the construction of this clique decomposition will be substantially more involved. We will borrow ideas from general work of Montgomery~\cite{Mon17} on fractional clique-decompositions of partite graphs (in particular, making use of ``gadgets'' to shift weight around), but certain simplifications are possible in the setting of triangle-decompositions. Also, instead of directly describing an \emph{explicit} fractional triangle-decomposition (as in \cite[Lemma~5.1]{KSSS22}, \cite[Lemma~4.2]{BGKLMO20}, and \cite{Mon17}), we construct our fractional triangle-decomposition as the limit of an iterative ``adjustment'' process. This was inspired by some related ideas of Bowditch and Dukes~\cite{BD19}.

\begin{proof}[Proof of \cref{lem:reg}]
For $v\in V(G)$ let $\mc{T}_v$ be the set of triangles of $\mc{T}$ involving $v$ and for $e\in E(G)$ let $\mc{T}(e)$ be the set of triangles of $\mc{T}$ containing $e$. Given $\phi: \mc{T}\to\mb{R}$, define the \emph{vertex-weight} of $v\in V(G)$ and \emph{edge-weight} of $e\in E(G)$ to be
\[\phi^{\mr{vtx}}(v) = \sum_{T\in\mc{T}_v}\phi(T),\qquad\phi^{\mr{edge}}(e) = \sum_{T\in\mc{T}(e)}\phi(T).\]
The total weight is $\phi^{\mr{sum}} = \sum_{T\in\mc{T}}\phi(T)$. We say $\phi$ is \emph{vertex-balanced} or \emph{$\alpha$-edge-balanced}, respectively, if
\[\phi^{\mr{vtx}}(v) = \frac{3\deg_G(v)}{2|E(G)|}\phi^{\mr{sum}}~\forall v\in V(G),\qquad\phi^{\mr{edge}}(e) = (1\pm\alpha)\frac{3}{|E(G)|}\phi^{\mr{sum}}~\forall e\in E(G).\]
Summing over the edges incident to each vertex shows that $0$-edge-balancedness implies vertex-balancedness.

Our goal is to produce $\phi_*:\mc{T}\to[0,1]$ with $\phi_*^{\mr{edge}}(e) = (1+O(n^{-1/3}))p^2qn/4$ for all $e$ (so in particular $\phi_*$ is $O(n^{-1/3})$-edge-balanced). This suffices to prove the lemma: if we take $\mc{T}'$ to be a random subcollection of $\mc{T}$ in which each $T\in\mc{T}$ included with probability $\phi_*(T)$ independently, then the Chernoff bound and union bound show that $\mc{T}'$ satisfies the desired properties with high probability.

So, for the rest of the proof we construct our desired edge-balanced weight function $\phi_*$. We do this in several steps: first, we build a vertex-balanced function by averaging over certain simple functions $\chi_{u,v}$. Then, we iteratively adjust this function to make it edge-balanced (our final edge-balanced function will be a fixed point of a certain contraction map). Finally, we divide by an appropriate constant to obtain our desired function $\phi_*:\mc{T}\to[0,1]$ with $\phi_*^{\mr{edge}}(e) = (1+O(n^{-1/3}))p^2qn/4$ for all edges $e$.

\medskip
\noindent \textit{Step 1: Constructing a near-vertex-balanced function.} For $j\in \{1,2,3\}$ and distinct $u,v\in V^j$, let $\mc{T}_{2,1,1}(u,v)$ be the set of copies of $K_{2,1,1}$ in $G$ containing $u,v$, such that both triangles of the copy of $K_{2,1,1}$ are in $\mc{T}$ (this can only happen when $u,v$ are in the size-2 part of the $K_{2,1,1}$). For $H\in\mc{T}_{2,1,1}(u,v)$ let $H(u),H(v)$ be the triangles involving $u$ and $v$ respectively.

For a vertex $u$, let
\[f_u = |\mc{T}_u|-\frac{3\deg_G(u)}{2|E(G)|}|\mc{T}|,\]
and note that for each $j\in \{1,2,3\}$, we have $\sum_{u\in V^j}f_u = 0$ since $G$ has the same number of edges between each pair of parts.

For $j\in \{1,2,3\}$ and distinct $u,v\in V^j$, define $\chi_{u,v}:\mc{T}\to\mb{R}$ by
\[\chi_{u,v}(T) = \frac{1}{|\mc{T}_{2,1,1}(u,v)|}\sum_{H\in\mc{T}_{2,1,1}(u,v)}(\mbm{1}_{T=H(u)}-\mbm{1}_{T=H(v)}),\]
and for $j\in \{1,2,3\}$, define $\phi_{V^j}:\mc{T}\to\mb{R}$ by
\[\chi_{V^j} = -\frac{1}{2n}\sum_{u,v\in V^j:u\ne v}(f_u-f_v)\chi_{u,v}.\]
We note some key facts about the vertex-weights of these functions: for any $w\in V(G)$ we have
\[\chi_{u,v}^{\mr{vtx}}(w) = \begin{cases}1&\text{if }w=u,\\-1&\text{if }w=v,\\0&\text{otherwise},\end{cases}\]
and thus, recalling that $\sum_{u\in V_i}f_u = 0$, we have
\[\chi_{V^j}^{\mr{vtx}}(w) = \begin{cases}-f_w&\text{if }w\in V^j,\\0&\text{if }w\notin V^j.\end{cases}\]

Let $\bs 1 : \mc T \to \mb R$ be the all-$1$ function and define $\phi_0 = \bs 1 +\chi_{V^1}+\chi_{V^2}+\chi_{V^3}$. Note that $\phi_0$ is vertex-balanced: for any $w\in V(G)$ we have
\[\phi_0^\mr{vtx}(w)=\sum_{T\in\mc{T}_w}(1+\chi_{V^1}(T)+\chi_{V^2}(T)+\chi_{V^3}(T)) = |\mc{T}_w| + 0 + 0 - f_w = \frac{3\deg_G(w)}{2|E(G)|}|\mc{T}|.\]

\medskip
\noindent\textit{Step 2: Defining ``adjuster'' functions.} 
For a $6$-cycle $J\subseteq G$ alternating between some $V^{j-1}$ and $V^j$, let $\mc{T}_{3,3,1}(J)$ be the set of vertices $v\in V^{j+1}$ which complete a triangle in $\mc T$ with each edge of $J$. For an edge $e\in G_{j-1,j}$, let $G(e)$ be the set of $6$-cycles $J\subseteq G$, alternating between $V^{j-1}$ and $V^j$, which contain $e$. 

For $e\in E(G)$, and a $6$-cycle
$J\in G(e)$, we define $\psi_{J,e} : \mc T \to \mb R$ by 
\[\psi_{J,e}(T) = \frac{1}{|\mc{T}_{3,3,1}(J)|}\sum_{v\in\mc{T}_{3,3,1}(J)}\sum_{e'\in J}(-1)^{d_J(e',e)}\mbm{1}_{T=e'\cup\{v\}},\]
where $d_J(e',e)$ is the minimum number of times $e'$ must be rotated around the cycle in either direction to coincide with $e$. For an edge $e\in E(G)$ between $V^{j-1}$ and $V^j$, define
\[\psi_e = \frac{1}{|G(e)|}\sum_{J\in G(e)}\psi_{J,e}.\]
We now claim that these two functions have zero vertex-weights (and can thus be used to modify edge-weights without modifying vertex-weights). Indeed, for $e\in E(G)$, $J\in G(e)$, $v\in \mc T_{3,3,1}(J)$ and any vertex $w\in V(G)$, we have
\[\sum_{T\in\mc{T}_w}\sum_{e'\in J}(-1)^{d_J(e',e)}\mbm{1}_{e'\cup\{v\}}(T) = 0.\]
(Basically, the idea is that this alternating sum around a $6$-cycle contributes $0$ to every vertex-weight.) It follows that $\psi_{J,e}^{\mr{vtx}}(w) = \psi_e^{\mr{vtx}}(w) = 0$, as claimed.

\medskip
\noindent\textit{Step 3: One-step adjustment of a vertex-balanced function}. For a function $\phi:\mc T\to \mb R$, define the edge-discrepancy function $\phi^{\mr{disc}}: E(G)\to\mb{R}$ by
\[\phi^{\mr{disc}}(e) = \phi^{\mr{edge}}(e)-\frac{3}{|E(G)|}\phi^{\mr{sum}}.\]
Note that $\phi$ is $0$-edge-balanced if and only if $\phi^\mr{disc}\equiv 0$, and note that for $v\in V^j$ and $i\neq j$, and any vertex-balanced $\phi$, we have
\begin{equation}\label{eq:edge-disc-average}
\sum_{\substack{e\in G_{i,j}}:\\v\in e}\phi^{\mr{disc}}(e) = \phi^{\mr{vtx}}(v) - \frac{3\deg_G(v)}{2|E(G)|}\phi^{\mr{sum}} + \frac{3(\deg_Gv-\deg_{V^i}(v))}{2|E(G)|}\phi^{\mr{sum}} = O(n^{2/3}\snorm{\phi}_\infty).
\end{equation}
For a vertex-balanced function $\phi:\mc T\to \mb R$ we define its one-step adjustment $A(\phi): \mc{T}\to\mb{R}$ by
\[A(\phi) = \phi - \sum_{e\in E(G)}\phi^{\mr{disc}}(e)\psi_e.\]
Note that $A(\phi)$ is again vertex-balanced, since the $\psi_e$ have vertex-weight zero.

\medskip
\noindent\textit{Step 4: Bounding discrepancy after adjustment.} We now analyze how edge-discrepancy changes between $\phi$ and $A(\phi)$. First, we claim that for any $e\in E(G)$ and $J\in G(e)$, and any $f\in E(G)$, we have
\[\psi_{J,e}^{\mr{edge}}(f) = \begin{cases}(-1)^{d_J(f,e)}&\text{if }f\in J,\\0&\text{otherwise}.\end{cases}\]
Indeed, this follows from the fact that for $e\in E(G)$ $J\in G(e)$, $v\in \mc T_{3,3,1}(J)$, and any $f\in E(G)$, we have
\[\sum_{T\in\mc{T}(f)}\sum_{e'\in J}(-1)^{d_J(e',e)}\mbm{1}_{T=e'\cup\{v\}}(T) = \begin{cases}(-1)^{d_J(f,e)}&\text{if }f\in J,\\0&\text{otherwise}.\end{cases}.\]
(Here we used that the edges between $v$ and a vertex of $J$ are zeroed out by the alternating subtraction.) It follows that for $e,f\in G_{j-1,j}$ we have
\[\psi_e^{\mr{edge}}(f) = \frac{c_0(e,f)-c_1(e,f)+c_2(e,f)-c_3(e,f)}{|G(e)|}\]
where $c_i(f,e)$ is the number of $6$-cycles $J\subseteq G_{j-1,j}$ including $f$ and $e$ with $d_J(e,f) = i$. Note that if $e,f$ are not between the same pair of parts then $\psi_f^{\mr{edge}}(e) = 0$, and that if $e = f$ then $\psi_e^{\mr{edge}}(f) = 1$.

Note that $\psi_e^\mr{sum}=0$ for all $e$, so for any $f\in G_{j-1,j}$ we have
\begin{align}
A(\phi)^{\mr{disc}}(f) &=\phi^\mr{disc}(f)-\sum_{e\in E(G)}\phi^{\mr{disc}}(e)\psi_e^{\mr{edge}}(f)
= -\sum_{\substack{e\in G_{j-1,j}\\e\neq f}}\phi^{\mr{disc}}(e)\psi_e^{\mr{edge}}(f)\notag\\
&= -\sum_{\substack{e\in G_{j-1,j}\\e\neq f}}\phi^{\mr{disc}}(e)\frac{-c_1(e,f)+c_2(e,f)-c_3(e,f)}{|G(e)|}\label{eq:phidisc}.
\end{align}
We simplify this expression by defining some further statistics. Suppose $e,f\in G_{j-1,j}$ are distinct. If $e,f$ share a vertex (which we denote $e \sim f$), let $g_1(e,f)$ be the number of extensions of $e\cup f$ to a $6$-cycle in $G_{j-1,j}$ (note $g_1(e,f)=c_1(e,f)$). If $e,f$ do not share a vertex but there are vertices $u\in e$ and $v\in f$ that are adjacent to each other in $G$ (which we denote $e\sim_{u,v} f$), then let $g_2(e,f,u,v)$ be the number of 6-cycles in $G_{j-1,j}$ containing the edges $e,uv,f$. If $e,f$ share no vertex (which we denote $e\not \sim f$), let $g_3(e,f)$ be the number of $6$-cycles $J\subseteq G_{j-1,j}$ in which $d_J(e,f)=3$ (i.e., $e$ and $f$ are on opposite sides of the 6-cycle). It will be convenient to write $g_1(e,e)=g_2(e,e,u,v) = g_3(e,e) = 0$ as well as $g_3(e,f) = 0$ when $f\sim e$. Finally, given a vector $\vec{a}\in\mb{R}^t$ let $\on{disc}(\vec{a}) = \sum_{j=1}^t|a_j-\mu(\vec{a})|$ for $\mu(\vec{a}) = (a_1+\cdots+a_t)/t$. We will use the key fact that if $b_1+\cdots+b_t = 0$ then
\begin{equation}
    |a_1b_1+\cdots+a_tb_t|\le \sum_{i=1}^t |(a_i-\mu(\vec{a}))b_i|+|\mu(\vec{a})||b_1+\cdots+b_t|\le\snorm{\vec{b}}_\infty\on{disc}(\vec{a})+\snorm{\vec{a}}_\infty|b_1+\cdots+b_t|.\label{eq:disc}
\end{equation}
Now, continuing from \cref{eq:phidisc}, using \cref{eq:disc} with \cref{eq:edge-disc-average}, we obtain
\begin{align*}
|A(\phi)^{\mr{disc}}(f)|&= \Bigg|\sum_{e\sim f}\frac{g_1(e,f)}{|G(e)|}\phi^{\mr{disc}}(e) - \sum_{u,v,e: e\sim_{u,v}f}\frac{g_2(e,f,u,v)}{|G(e)|}\phi^{\mr{disc}}(e) + \sum_{e\not\sim f}\frac{g_3(e,f)}{|G(e)|}\phi^{\mr{disc}}(e)\Bigg|\\
&\le\snorm{\phi^{\mr{disc}}}_\infty\left(\sum_{v\in f}\on{disc}\Bigg(\left(\frac{g_1(e,f)}{|G(e)|}\right)_{\substack{e\in G_{j-1,j}:\\v\in e}}\Bigg) + \sum_{\substack{u,v:\,v\in f,\\ uv\in G_{j-1,j}}}\on{disc}\Bigg(\left(\frac{g_2(e,f,u,v)}{|G(e)|}\right)_{\substack{e\in G_{j-1,j}:\\u\in e}}\Bigg)\right.\\
&\qquad\qquad\qquad\qquad\qquad+\left.\vphantom{\sum_{\substack{u,v:\,v\in f,\\ uv\in G_{j-1,j}}}\on{disc}\Bigg(\left(\frac{g_2(e,f,u,v)}{|G(e)|}\right)_{\substack{e\in G_{j-1,j}:\\u\in e}}\Bigg)}
\sum_{u\in V_i\setminus e}\on{disc}\Bigg(\left(\frac{g_3(e,f)}{|G(e)|}\right)_{\substack{e\in G_{j-1,j}:\\u\in e}}\Bigg)\right) + A\cdot O(n^{2/3}\snorm{\phi}_\infty),
\end{align*}
where \[A=\sum_{v\in f}\Bigg\|\left(\frac{g_1(e,f)}{|G(e)|}\right)_{\substack{e\in G_{j-1,j}:\\v\in e}}\Bigg\|_\infty + \sum_{\substack{u,v:\,v\in f,\\ uv\in G_{j-1,j}}}\Bigg\|\left(\frac{g_2(e,f,u,v)}{|G(e)|}\right)_{\substack{e\in G_{j-1,j}:\\u\in e}}\Bigg\|_\infty+\sum_{u\in V_i\setminus e}\Bigg\|\left(\frac{g_3(e,f)}{|G(e)|}\right)_{\substack{e\in G_{j-1,j}:\\u\in e}}\Bigg\|_\infty.\]

Now we use the quasirandomness hypotheses in the lemma statement, which we can apply repeatedly to count copies of any graph $H$ in $G_{j-1,j}$ extending any given set of vertices and edges (we are assuming $p\ge n^{-1/12}$, so the contribution from ``degenerate 6-cycles'' with repeated vertices is negligible). For distinct edges $e\sim f$ we have
\[\frac{g_1(e,f)}{|G(e)|} = \frac{(1\pm O(\xi)) p^4n^3}{(1\pm O(\xi))^4p^5n^4} = \frac{1\pm O(\xi)}{pn}.\]
For $e\sim_{u,v} f$ we have
\[\frac{g_2(e,f,u,v)}{|G(e)|} = \frac{(1\pm O(\xi))p^3n^2}{(1\pm O(\xi))p^5n^4} = \frac{1\pm O(\xi)}{p^2n^2},\]
and for $e\not \sim f$ we have
\[\frac{g_3(e,f)}{|G(e)|} = \frac{(1\pm O(\xi))p^4n^2}{(1\pm O(\xi))p^5n^4} = \frac{1\pm O(\xi)}{pn^2}.\]
We deduce
\begin{align*}
|A(\phi)^{\mr{disc}}(e)|&\le\snorm{\phi^{\mr{disc}}}_\infty\bigg(O(pn)\cdot\frac{O(\xi)}{pn} + O(p^2n^2)\cdot \frac{O(\xi)}{p^2n^2} + O(pn^2)\cdot \frac{O(\xi)}{pn^2}\bigg) + O(n^{2/3}\snorm{\phi}_\infty)\\
&\lesssim \xi\snorm{\phi^{\mr{disc}}}_\infty + n^{2/3}\snorm{\phi}_\infty.
\end{align*}
That is to say, if $\xi$ is sufficiently small, then repeated application of the ``adjustment'' map $A$ reduces the discrepancy of $\phi$, at least until it reaches size $n^{2/3}\snorm{\phi}_\infty$ (which in our case will be order $n^{2/3}$).

\medskip
\noindent\textit{Step 5: Final estimates.} We need a few further bounds before completing the proof (we will use these to ensure that our final weight function $\phi_*$ has its values in $[0,1]$).

First, we consider $\phi_0$. Recalling the definition of $\chi_{V^j}$, for each $j\in \{1,2,3\}$ we have
\[|\chi_{V^j}(T)|\lesssim\frac{1}{n}\sup_{u\in V(G)}|f_u|\cdot\sum_{u,v\in V^j}\frac{1}{|\mc T_{2,1,1}(u,v)|}\sum_{H\in \mc T_{2,1,1}(u,v)}\mbm 1_{T\in \{H(u),H(v)\}}.\]
By the assumptions of the lemma, for any $u$ we have $|\mc T_u|=(1\pm O(\xi))(pn)(p^2qn)$ (considering the edges containing $u$, and then considering, for each such edge, the number of triangles including that edge), and therefore $|f_u|=O(\xi p^3 q n^2)$. For any $u,v$ we have $|\mc T_{2,1,1}(u,v)|\ge (1- \xi) (p^2 n) (C^{-1} p^3 n)$ (considering all length-2 paths between $u,v$, and then considering, for each such 2-path, the number of triangles which extend both the edges that 2-path to triangles). Also, $\sum_{u,v\in V^j}\sum_{H\in \mc T_{2,1,1}(u,v)}\mbm 1_{T\in \{H(u),H(v)\}}$ can be interpreted as a count of copies of $K_{2,1,1}$ which contain $T$; this is always at most $Cp^2n$.

So, recalling the definition of $\phi_0$ and the all-1 function $\bs 1:\mc T\to \mb R$, we have
\begin{equation}
    \snorm{\phi_0-\bs 1}_\infty = \snorm{\chi_{V^1}+\chi_{V^2}+\chi_{V^3}}_\infty \lesssim C^2\xi.\label{eq:phiminusone}
\end{equation}
We also upper-bound $\snorm{\phi_0^{\mr{disc}}}_\infty$: first recall that every edge is in $(1\pm \xi)p^2qn$ triangles, so it follows from \cref{eq:phiminusone} that $\snorm{(\phi_0-\bs 1)^\mr{disc}}_\infty\lesssim C^2\xi p^2 n$. We then again use that every edge is in $(1\pm \xi)p^2qn$ triangles to see that $\snorm{\bs 1^{\mr{disc}}}_\infty\lesssim\xi p^2qn$, from which we deduce that $\snorm{\phi_0^{\mr{disc}}}_\infty\lesssim C^2\xi p^2n$.

Now, we prove a bound on $\snorm{A(\phi)-\phi}_\infty$ for vertex-balanced $\phi$. Say a \emph{6-pyramid} $\mc J$ is the set of all triangles containing $v$ and a vertex of $J$, for some 6-cycle $J\in G_{j-1,j}$ and some vertex $v\in \mc T_{3,3,1}(J)$. Let $N_T$ be the number of 6-pyramids (with respect to the entire graph $G$ and set of triangles $\mc T$) containing a particular triangle $T$. Note that for any $e\in E(G)$ we have

\[
|\psi_e(T)|\le\frac{1}{|G(e)| \inf_{J \in G(e)} |\mc T_{3,3,1}(J)| }\sum_{\mc J}\mbm{1}_{T\in \mc J}\le \frac{1}{C^{-1}p^6n\cdot (1-O(\xi))^4p^5n^4}\sum_{\mc J}N_T
\le\frac{2C}{p^{11}n^5}N_T,
\]
where the sums are over all 6-pyramids $\mc J$ whose underlying 6-cycle contains $e$.

To specify a 6-pyramid containing $T$, whose underlying 6-cycle contains a particular edge $f\in T$, we need to specify four additional vertices. The first three of these vertices must each extend a particular edge to a triangle, and then the final vertex must extend two different edges to triangles. So, by the  assumptions of the lemma (and accounting for the three choices of $f$) we have $N_T\le 3(Cp^2 n)^3(Cp^3)\le C^4 p^9 n^4$.

Recalling the definition of $A(\phi)$, it follows that for any vertex-balanced $\phi$ we have
\[|A(\phi)(T)-\phi(T)|\le \snorm{\phi^{\mr{disc}}}_\infty\cdot6\cdot \frac{2C}{p^{11}n^5}N_T\lesssim \frac{C^5}{p^2n}\snorm{\phi^{\mr{disc}}}_\infty.\]
(here we are using that each 6-pyramid contributes to $\psi_e$ for 6 different $e$). We summarize
\[\snorm{A(\phi)^{\mr{disc}}}_\infty\lesssim\xi\snorm{\phi^{\mr{disc}}}_\infty+n^{2/3}\snorm{\phi}_\infty,\quad\snorm{A(\phi)-\phi}_\infty\lesssim\frac{C^5}{p^2n}\snorm{\phi^{\mr{disc}}}_\infty\]
for any vertex-balanced $\phi$, and
\[\snorm{\phi_0-\bs 1}_\infty\lesssim C^2\xi,\quad \snorm{\phi_0^{\mr{disc}}}_\infty\lesssim C^2\xi p^2n.\]

\medskip
\noindent\textit{Step 6: Iteration.} Let $\phi_k=A^{(k)}(\phi_0)$ be the result of iterating the ``adjustment map'' $k$ times. We will show that $\phi_k$ for $k$ near $(\log n)^2$ will suffice for our purposes.
First, we claim that $\snorm{\phi_k-\mbf{1}}_\infty\le 1/10$ for all $0\le k\le(\log n)^2$. We use induction: recall that $\snorm{\phi_0-\mbf{1}}_\infty\lesssim C^2\xi$ and that $\xi = C^{-8}$, which implies the result for $k = 0$ (assuming $C$ is sufficiently large). If the statement is true up to $k$, then we deduce
\[\snorm{\phi_i^{\mr{disc}}}_\infty\le(O(\xi))^k\snorm{\phi_0^{\mr{disc}}}_\infty + 2n^{2/3}\]
by iteration, and hence
\[\snorm{\phi_{k+1}-\phi_0}_\infty\le\sum_{i=0}^k\snorm{\phi_{i+1}-\phi_i}_\infty\lesssim\frac{C^5}{p^2n}(\snorm{\phi_0^{\mr{disc}}}_\infty + kn^{2/3})\lesssim C^7\xi.\]
Recalling that $\xi = C^{-8}$, for sufficiently large $C$ we find $\snorm{\phi_{k+1}-\phi_0}_\infty\le C^7\xi$. Combining with $\snorm{\phi_0-\mbf{1}}_\infty\lesssim C^2\xi$ and assuming $C$ large proves the result for $k+1$, completing the induction.

Next, this means that
\[\snorm{\phi_k^{\mr{disc}}}_\infty\le(O(\xi))^k\snorm{\phi^{\mr{disc}}}_\infty + 2n^{2/3}\]
by iteration, hence for $k = (\log n)^2$ we obtain $\snorm{\phi_k^{\mr{disc}}}_\infty\le 3n^{2/3}$.

Finally, fix some edge $e\in E(G)$. Since $\phi_\infty$ takes values in $[1\pm 1/10]$, it follows that $\alpha\coloneqq\phi_\infty^{\mr{edge}}(e) = (1\pm 1/10)(1\pm\xi)p^2qn$ for all edges $e$ and that $\beta \coloneqq 4\alpha / (p^2qn) = (1 \pm 1/10)(1 \pm \xi)4$. Therefore $\phi_* = \phi_\infty / \beta$ is a suitable function with $\phi_*^\mr{edge}=p^2q n/4+O(n^{2/3})$, as desired.
\end{proof}

\subsection{Counting intercalate-free Latin squares}\label{sub:counting-intercalate-free}

In this section we perform the necessary calculations to prove \cref{thm:pasch-free}. An analogous result for high-girth Steiner triple systems was proved by the authors in \cite{KSSS22}.

We can interpret the random construction described in \cref{sub:nutshell} as producing a triangle-decomposition of $K_{N,N,N}$ in which the set of triangles is \emph{ordered}. Indeed, the triangles in $\mc I$ arising from the initial sparsification process come first in the ordering (in the order they are chosen in the process), and then the remaining triangles come afterwards (in some arbitrary order; this order will not be important for us). We will lower-bound the number of ordered triangle-decompositions that can arise from our random construction, and then divide by $(N^2)!$ to obtain a lower bound on the number of intercalate-free Latin squares. This type of argument was first used by Keevash~\cite{Kee18} to count Steiner triple systems.

The analysis in \cite[Section 9]{KSSS22} implies that whp the initial sparsification process succeeds in constructing a partial Latin square $\mc I$ of girth greater than $6$ (i.e., an intercalate-free partial Latin square) with $M := \lfloor (1-N^{-\nu})N^2 \rfloor$ triangles, for some absolute constant $\nu>0$. Furthermore, the analysis shows that whp for every $0 \leq t \leq M$, the number of choices available to the process at step $t$ is $(1 \pm N^{-\nu}) A(t)$, with
\[
A(t) = N^3 \left( 1 - \frac{t}{N^2} \right)^3 \exp \left( - \frac{t^3}{N^6} \right).
\]
We say an outcome of the initial sparsification process is \emph{good} if it satisfies the above properties.

\begin{remark}
There is a compelling heuristic explanation for the above formula. Denote by $\mc C(t)$ the set of triangles chosen in the first $t$ steps of the process. Observe that a triangle $T$ is available to be selected only if none of its edges has been covered by $\mc C(t)$ and there does not exist an intercalate containing $T$ with the other three triangles in $\mc C(t)$. Note that the density of edges not covered by $\mc C(t)$ is approximately $(1-t/N^2)$. Thus, if the graph of uncovered edges is pseudorandom, then there are approximately $N^3 \left( 1-t/N^2 \right)^3$ triangles with all edges uncovered. Additionally, if we fix an uncovered triangle $T$ then the number of intercalates containing $T$ (in $K_{N,N,N}$) is approximately $N^3$. If we think of $\mc C(t)$ as a random set of triangles of the same density $|\mc C(t)|/N^3 = t/N^3$ then the expected number of intercalates containing $T$ whose other three triangles are in $\mc C(t)$ is approximately $N^3 \left( t/N^3 \right)^3 = t^3\!/N^6$. Thus, the factor $\exp \left( - t^3\!/N^6 \right)$ captures the heuristic that the number of intercalates containing $T$, whose other three triangles are in $\mc C(t)$, has an approximate Poisson distribution with parameter $t^3\!/N^6$.
\end{remark}

Now, in our initial sparsification process, we choose a uniformly random triangle at each step (among the valid choices at that step). The probability that this process produces a specific good ordered set of triangles is therefore $\prod_{t=0}^{M-1}((1 \pm N^{-\nu}) A(t))^{-1}$.
Since the process produces a good ordered set of triangles with probability $1-o(1)$, it follows that the number of good outcomes that can arise is
\begin{align*}
(1-o(1))\prod_{t=0}^{M-1} (1 \pm N^{-\nu})A(t) & = \left((1 \pm N^{-\nu})N^3\right)^{N^2} \exp \left( \sum_{t=0}^{M-1} \log \left( \left( 1- \frac{t}{N^2} \right)^3 \exp \left( - \frac{t^3}{N^6} \right) \right) \right)\\
& = \left((1 \pm N^{-\nu/2})N^3\right)^{N^2} \exp \left( N^2 \int_0^1 \log\left( (1-x)^3 \exp \left( -x^3 \right) \right) dx \right)\\
& = \left((1 \pm N^{-\nu/2})N^3\right)^{N^2} \exp \left( - \frac{13N^2}{4} \right) = \left( \left( 1 \pm N^{-\nu/2} \right) \frac{N^3}{e^{13/4}}  \right)^{N^2}.
\end{align*}
Dividing this expression by $(N^2)!=((1+o(1))N^2/e)^{N^2}$ yields the desired result, once we additionally note that whp such a choice of triangles can be completed to a full high-girth Latin square (due to the prior analysis).

\bibliographystyle{amsplain0.bst}
\bibliography{main.bib}

\end{document}